\documentclass[11pt, amsfonts]{amsart}

\usepackage{amsmath,amssymb,amsthm,color,enumerate}
\usepackage[all]{xy}
\usepackage{tikz}
\usepackage{todonotes}
\usepackage{url} 
\usepackage{hyperref} 

\numberwithin{equation}{section}

\SelectTips{eu}{12}

\newtheorem{thm}{Theorem}[section]

\newtheorem{theorem}[thm]{Theorem}
\newtheorem{proposition}[thm]{Proposition}
\newtheorem{lemma}[thm]{Lemma}
\newtheorem{corollary}[thm]{Corollary}

\theoremstyle{definition}

\newtheorem{definition}[thm]{Definition}
\newtheorem{example}[thm]{Example}

\theoremstyle{remark}

\newtheorem{remark}[thm]{Remark}

\newcommand{\bt}{{\, \bowtie}}

\newcommand{\domcpx}{\mathcal D}
\newcommand{\nbdcpx}{\mathcal N}
\newcommand{\indepcpx}{\mathcal I}
\newcommand{\boxcpx}{\mathcal B}

\newcommand{\ZZ}{\mathbb{Z}}
\newcommand{\conn}{{\rm conn}}

\newcommand{\hdim}{\textrm{\rm h-dim}}

\newcommand{\pt}{\mathrm{pt}}
\renewcommand{\bt}{{\, \bowtie}}

\title[Dominance complexes and neighborhood complexes]{Dominance complexes, neighborhood complexes and combinatorial Alexander duals}

\author{Takahiro Matsushita}
\address{Department of Mathematical Sciences, Faculty of Science, Shinshu University, Matsumoto, Nagano 390-8621, Japan}
\email{matsushita@shinshu-u.ac.jp}

\author[S. Wakatsuki]{Shun Wakatsuki}
\address{Graduate School of Mathematics, Nagoya University, Furocho, Chikusaku,
Nagoya, 464-8602, Japan}
\email{shun.wakatsuki@math.nagoya-u.ac.jp}



\begin{document}

\baselineskip.525cm

\maketitle

\begin{abstract}
We show that  the dominance complex $\domcpx(G)$ of a graph $G$ coincides with the combinatorial Alexander dual of the neighborhood complex $\nbdcpx(\overline{G})$ of the complement of $G$. Using this, we obtain a relation between the chromatic number $\chi(G)$ of $G$ and the homology group of $\domcpx(G)$. We also obtain several known results related to dominance complexes from well-known facts of neighborhood complexes. After that, we suggest a new method for computing the homology groups of the dominance complexes, using independence complexes of simple graphs. We show that several known computations of homology groups of dominance complexes can be reduced to known computations of independence complexes. Finally, we determine the homology group of $\domcpx(P_n \times P_3)$ by determining the homotopy types of the independence complex of $P_n \times P_3 \times P_2$.
\end{abstract}


\section{Introduction}

\subsection{Background and the first result}
The initial purpose of this paper is to reveal a close relation between two simplicial complexes that appear in topological combinatorics. One is the dominance complex, and the other is the neighborhood complex.

We first recall these complexes. Let $G = (V(G), E(G))$ be a finite simple graph. A subset $\sigma$ of $V(G)$ is \emph{dominating in $G$} if every vertex $v$ of $G$ is contained in $\sigma$ or adjacent to an element in $\sigma$. The \emph{dominance complex $\domcpx(G)$ of $G$} is the simplicial complex consisting of subsets of $V(G)$ whose complements are dominating. The dominance complex of a graph was introduced in Ehrenborg and Hetyei \cite{EH}, and has been studied by several authors (see for example \cite{Klivans, MT2, MT1, Matsushita6, Taylan}).

Next we recall the neighborhood complex, which is one of the most important simplicial complexes in topological combinatorics. See Section 2 for the definition. The neighborhood complex was introduced by Lov\'asz \cite{Lovasz} in his celebrated proof of Kneser's conjecture. He showed that there is a close relation between a certain homotopy invariant of the neighborhood complex $\nbdcpx(G)$ and the chromatic number of $G$, and determined the chromatic numbers of Kneser's graphs. After that, the neighborhood complex and its generalizations have been studied extensively by many authors (see for example \cite{BK, BK2, Csorba, Dochtermann, Kozlov1, Matsushita2, Matsushita3, Matsushita5, Matsushita4, ST, Walker, Wrochna1}).


The following result is the starting point of our discussion. In fact, the dominance complex of a graph is essentially an equivalent object to the neighborhood complex of the complement. All necessary definitions will be given in Section 2.

\begin{theorem} \label{main theorem}
Let $G$ be a finite simple graph. Then
\[ \domcpx(G)^\vee = \nbdcpx(\overline{G}).\]
Here $\overline{G}$ denotes the complement of $G$, and $K^\vee$ denotes the combinatorial Alexander dual of a simplicial complex $K$.
\end{theorem}

This theorem will be proved in Section 3. As we will see, the proof is straightforward and short, and in fact very close considerations have already appeared in the literature (see Remark \ref{remark 3.1} for details). However, to the best of our knowledge, we have not found the above formulation that uses all of dominance complexes, neighborhood complexes and combinatorial Alexander duals.


\subsection{Applications}
We have mentioned that a certain homotopy invariant of the neighborhood complex is related to the chromatic number. In fact, by Theorem \ref{main theorem}, we can obtain a relation between the homology groups of the dominance complex and the chromatic number. This is the next theorem we will mention. To state the theorem precisely, we introduce the following concept: For a finite simplicial complex $X$ and for a commutative ring $R$ with unit, we define the \emph{\(R\)-homological dimension} $\hdim_R(X)$\label{index:homologicalDimension} by
\[ \hdim_R(X) = \sup \{ n \; | \; \tilde{H}_n(X ; R) \ne 0\} \in \ZZ \cup \{ - \infty\}.\]
The following theorem will be proved in Subsection 4.1.

\begin{theorem} \label{theorem 2}
Let $G$ be a graph with $n$ vertices. Then 
\[ \chi(\overline{G}) \ge n - \hdim_{\ZZ_2}(\domcpx(G)) - 2,\]
where \(\ZZ_2 = \ZZ / 2\ZZ\).
\end{theorem}

Next we mention two other applications of Theorem \ref{main theorem}.

Recall that a vertex cover of a graph $G$ is a subset $\sigma$ of $V(G)$ such that every edge contains an element of $\sigma$. The \emph{vertex cover number $\tau(G)$ of $G$}\label{index:vertexCoverNumber} is the smallest cardinality of a vertex cover of $G$. In the previous work \cite{Matsushita6}, the author obtained an inequality between $\tau(G)$ and the $\ZZ_2$-homological connectivity of the dominance complex $\domcpx(G)$. Theorem \ref{main theorem} provides an alternative proof of this fact (Subsection 4.2).

For the next application, we consider the dominance complex of a chordal graph. Recall that a graph $G$ said to be \emph{chordal} if there is no induced $n$-cycle with $n \ge 4$. Clearly, every forest is chordal. Ehrenborg and Hetyei \cite{EH} showed that the dominance complex of a forest is homotopy equivalent to a sphere (see also Marietti and Testa \cite{MT1}). This result was generalized to chordal graphs by Taylan \cite{Taylan}. She showed that the dominance complex of a chordal graph is homotopy equivalent to a sphere. Using Theorem \ref{main theorem} and the fold lemma for neighborhood complexes (Theorem \ref{theorem fold neighborhood}), we provide an alternative proof of this result.

\subsection{Method for computing homology}

Theorems \ref{main theorem} and \ref{theorem 2} suggest that it is important to compute the homology groups of the dominance complex when considering the graph coloring problem of the complement of a graph. Unfortunately, there are few methods for computing homology groups of dominance complexes, and few examples of dominance complexes of graphs whose homology groups are actually determined. In Sections $5$ and $6$, we propose a new method to compute the homology group of the dominance complex, using the independence complex of a finite simple graph.

Recall that a subset $\sigma$ of $V(G)$ is said to be \emph{independent in $G$} if $\sigma$ contains no edge in $G$. The \emph{independence complex $\indepcpx(G)$ of $G$}\label{index:independenceComplex} is the simplicial complex consisting of independent sets in $G$. The dominance complex and independence complex are related as follows: First, we define the graph $G^\bt$ by $V(G^\bt) = V(G) \times \{ + , -\}$ and
\[ E(G) = \{ \{ (v, + ), (w, -) \} \; | \; w \in N_G[v] \}.\]
Here $N_G[v]$ denotes the closed neighborhood of $v$ (see Section 2). The first author showed that there is a homotopy equivalence
\[ \Sigma \big( \domcpx(G)^\vee \big) \simeq \indepcpx(G^\bt).\]
See Corollary 4 of \cite{Matsushita6}. Here $\Sigma$ denotes the suspension. Thus, using the Alexander duality theorem (see Theorem \ref{theorem alexander} or \cite{BT}), we can compute the homology groups of $\domcpx(G)$ from the homology groups of $\indepcpx(G^\bt)$. In Section 5, we will see that several results of the homology groups of the dominance complex can be deduced from known results of the independence complexes. For example, the homology groups of the dominance complex of a cycle, which was actually computed in Taylan \cite{Taylan}, can be deduced from the study of the independence complexes of certain grid graphs studied by Okura \cite{Okura} and Thapper \cite{Thapper}. In Subsection 5.4, we obtain methods (Theorem \ref{theorem 1-connected criterion} and Corollary \ref{corollary easy criterion}) for determining if $\domcpx(G)$ is $1$-connected. We also show a relation between the minimum degree of $G$ and the connectivity of $\domcpx(G)$ (Proposition \ref{proposition connectivity}). Using Theorem \ref{theorem 1-connected criterion}, we determine the homotopy types of dominance complexes by their homology groups (see Theorem \ref{theorem 1-connected}). As a new result, we determine the homotopy type of $\domcpx(P_n \times P_2)$ from the homotopy type of the independence complex $\indepcpx (P_n \times C_4)$ by this method (see Example \ref{example PnxP2} and Theorem \ref{theorem 1-connected}).


Finally, in Section 6 we compute the homology groups of the dominance complex of the grid graph $P_n \times P_3$ for all $n$ by determining the homotopy type of the independence complex of $(P_n \times P_3)^\bt \cong P_n \times P_3 \times P_2$.

\subsection{Organization}
This paper is organized as follows. In Section 2, we review definitions and facts related to Theorem \ref{main theorem}. In Section 3, we prove Theorem \ref{main theorem}. In Section 4, we provide several applications of Theorem \ref{main theorem} including Theorem \ref{theorem 2}. In Section 5, we determine the homotopy types of dominance complexes in some examples, using $\indepcpx(G^\bt)$. In Section 6, we compute the homology groups of $\domcpx(P_n \times P_3)$ by determining the homotopy type of the independence complex of $(P_n \times P_3)^\bt \cong P_n \times P_3 \times P_2$.


In this paper, various complexes and invariants of graphs will appear, and we list the notations in Figure \ref{figure:listOfNotations} for the reader's convenience.

\begin{figure}[h]
  \caption{List of notation}
  \label{figure:listOfNotations}
  \begin{center}
    \begin{tabular}{|c|p{8 cm} c|}
      \hline
      \(V(G)\) & vertex set & \pageref{index:vertexSet} \\
      \(E(G)\) & edge set & \pageref{index:edgeSet} \\
      \(\chi(G)\) & chromatic number & \pageref{index:chromaticNumber} \\
      \(\alpha(G)\) & independence number & \pageref{index:independenceNumber} \\
      \(\tau(G)\) & vertex cover number & \pageref{index:vertexCoverNumber} \\
      \(\omega(G)\) & clique number & \pageref{index:cliqueNumber} \\
      \(\domcpx(G)\) & dominance complex & \pageref{index:dominanceComplex} \\
      \(\nbdcpx(G)\) & neighborhood complex & \pageref{index:neighborhoodComplex} \\
      \(\indepcpx(G)\) & independence complex & \pageref{index:independenceComplex} \\
      \(\boxcpx(G)\) & box complex & \pageref{index:boxComplex} \\
      \(G^\bt\) & a construction of a graph & \pageref{index:bowtie} \\
      \(\overline{G}\) & complement of a graph $G$ & \pageref{index:graphComplement} \\
      \(K^\vee\) & combinatorial Alexander dual & \pageref{index:AlexanderDual} \\
      \(N_G[v]\) & closed neighborhood & \pageref{index:closedNeighborhood} \\
      \(N_G(v)\) & open neighborhood & \pageref{index:openNeighborhood} \\
      \(\conn_R(X)\) & $R$-homological connectivity & \pageref{index:connectivity} \\
      \(\hdim_R(X)\) & $R$-homological dimension & \pageref{index:homologicalDimension} \\
      \hline
    \end{tabular}
  \end{center}
\end{figure}

\section{Preliminaries}

Here we review several definitions and facts we need in the proof of Theorem \ref{main theorem}. The known facts of the independence complexes we need are postponed to Subsection 5.1. We refer the reader to \cite{KozlovBook} for the terminology of simplicial complexes in topological combinatorics.

\subsection{Graphs}
A \emph{(finite simple) graph} $G$ is a pair $(V,E)$ consisting of a finite set $V$ and a subset $E$ of the set of $2$-element subsets of $V$. We call $V$ the \emph{vertex set of $G$} and $E$ the \emph{edge set of $G$}. We write $V(G)$\label{index:vertexSet} and $E(G)$\label{index:edgeSet} to indicate the vertex set and edge set of $G$, respectively. We say that $v$ and $w$ are \emph{adjacent in $G$} if $\{ v, w \} \in E(G)$. We write $v \sim_G w$ to mean that $v$ and $w$ are adjacent in $G$.

The \emph{complement of a graph $G$}\label{index:graphComplement} is the graph $\overline{G}$ defined by $V(\overline{G}) = V(G)$ and $E(\overline{G}) = \binom{V}{2} - E(G)$. Namely, $v$ and $w$ in $V(G)$ are adjacent in $\overline{G}$ if and only if $v$ and $w$ are not adjacent in $G$.

For a vertex $v$ of $G = (V, E)$, the \emph{open neighborhood of $v$ in $G$}\label{index:openNeighborhood}, denoted by $N_G(v)$, is the set of vertices that are adjacent to $v$ in $G$. The \emph{neighborhood complex $\nbdcpx(G)$ of $G$}\label{index:neighborhoodComplex} is the abstract simplicial complex whose underlying set is $V(G)$ and whose simplices are subsets of $V$ contained in an open neighborhood of a vertex. Set $N_G[v] = N_G(v) \cup \{ v\}$ and call it the \emph{closed neighborhood of $v$ in $G$}\label{index:closedNeighborhood}. We often write $N(v)$ (or $N[v]$) instead of $N_G(v)$ (or $N_G[v]$, respectively) when it is not necessary to mention $G$ precisely.

A subset $\sigma$ of $V(G)$ is \emph{dominating in $G$} if $V(G) = \bigcup_{v \in V(G)} N_G[v]$. The \emph{dominance complex $\domcpx(G)$ of $G$}\label{index:dominanceComplex} is the simplicial complex whose underlying set is $V(G)$ with simplices the subsets of $V(G)$ whose complements are dominating in $G$.
 We will freely use the following obvious lemma:

\begin{lemma} \label{lemma free}
Let $G$ be a graph and $\sigma$ be a subset of $V(G)$. Then $\sigma \in \domcpx(G)$ if and only if there is no $x \in V(G)$ such that $N_G[x] \subset \sigma$.
\end{lemma}

\subsection{Combinatorial Alexander dual}
An \emph{(abstract) simplicial complex} is a pair consisting of a finite set $S$, called the \emph{underlying set}, and a subset $K \subset S$ such that $\sigma \in \Delta$ and $\tau \subset \sigma$ imply $\tau \in K$. We often abbreviate to mention the underlying set, and simply write as ``$K$ is a simplicial complex.'' A \emph{vertex of $K$} is an element $v$ of $S$ such that $\{ v\} \in K$. The \emph{vertex set of $K$} is the set consisting of the vertices of $K$. An element $v \in S$ such that $\{ v\} \not\in K$ is called a \emph{ghost vertex}.

Let $K$ be a simplicial complex whose underlying set is $S$. The \emph{combinatorial Alexander dual $K^\vee$ of $K$}\label{index:AlexanderDual} is the simplicial complex defined as follows: The underlying set of $K^\vee$ is $S$. A subset $\sigma$ of $S$ is a face of $K^\vee$ if and only if $S - \sigma$ is not a face of $K$. It is straightforward to see that $(K^\vee)^\vee = K$.

The following theorem is a fundamental result of combinatorial Alexander duals.

\begin{theorem}[Combinatorial Alexander duality theorem, see \cite{BT}] \label{theorem alexander}
Let $R$ be a commutative ring with unit, and $K$ a simplicial complex with a ground set of the size $n$. Then
\[ \tilde{H}_i(K ; R) \cong \tilde{H}^{n-i-3}(K^\vee ; R).\]
\end{theorem}

Next we relate the homological dimension to the connectivity.
We define the \emph{$R$-homological connectivity\label{index:connectivity} $\conn_{R}(X)$ of $X$} by
\begin{equation}
  \conn_R(X) = \sup\{n \mid \tilde{H}_i(X; R) = 0 \text{ for all } i \le n\}
  \in \ZZ\cup \{\infty\}.
\end{equation}
The following lemma immediately follows from Theorem \ref{theorem alexander} and the universal coefficient theorem (see Subsection 3.1 of \cite{Hatcher}).

\begin{lemma}
  \label{lemma conn hdim}
  If \(R\) is a field, then
  \begin{equation}
    \conn_R(K) = n - 4 - \hdim_R(K^\vee).
  \end{equation}
\end{lemma}

\subsection{Some topological facts}
\label{subsection topological facts}

Here we provide the following topological facts which will be frequently used in this paper. For a topological space $X$, let $\Sigma X$ be the suspension of $X$.

\begin{proposition} \label{proposition suspension}
Let $X$ be a space, $R$ a commutative ring with unit and $i$ a non-negative integer. Then there is the following isomorphism:
\[ H_{i+1}(\Sigma X ; R) \cong H_i(X ; R).\]
\end{proposition}
The above proposition can be immediately proved by using the Mayer-Vietoris exact sequence \cite[Section 2.2]{Hatcher}.

\begin{proposition} \label{proposition homology of wedges of spheres}
Let $X$ be a topological space homotopy equivalent to
\[ \bigvee_{k \ge 0} \Big( \bigvee_{n_k} S^{k} \Big).\]
Then for a commutative ring $R$ with unit and for a non-negative integer $k$, there are the following isomorphisms:
\[ \tilde{H}_k(X; R) \cong \tilde{H}^k(X ; R) \cong R^{n_k}.\]
\end{proposition}
\begin{proof}
  We have \(\tilde{H}_k(X;\mathbb{Z}) \cong {\mathbb{Z}}^{n_k}\) by \cite[Corollaries 2.14 and 2.25]{Hatcher}.
  Then the universal coefficient theorems \cite[Theorem 3A.3, Section 3.1]{Hatcher} complete the proof.
\end{proof}

\section{Proof of Theorem \ref{main theorem}}

In this section, we prove Theorem \ref{main theorem}. As we will see in Remark \ref{remark 3.1}, similar observations have been made by several other authors.

\begin{proof}[Proof of Theorem \ref{main theorem}]
The underlying sets of $\domcpx(G)^\vee$ and $\nbdcpx(\overline{G})$ are both $V(G)$. Thus it suffices to show that for every subset $\sigma$ of $V(G)$, $\sigma$ is a simplex of $\domcpx(G)^\vee$ if and only if $\sigma$ is a simplex of $\nbdcpx(\overline{G})$. Consider the following conditions related to a subset $\sigma$ of $V(G)$:
\begin{enumerate}
\item $\sigma$ is a simplex of $\domcpx(G)^\vee$

\item $V(G) - \sigma$ is not a face of $\domcpx(G)$.

\item $\sigma$ is not dominating in $G$.

\item There is $v \in V(G)$ such that $N_G[v] \cap \sigma = \emptyset$.

\item There is $v \in V(G)$ such that $\sigma \subset N_{\overline{G}}(v)$.

\item $\sigma$ is a simplex in $\nbdcpx(\overline{G})$.
\end{enumerate}
The equivalence $(i) \Leftrightarrow (i + 1)$ is clear for $i = 1, \cdots, 5$. This completes the proof.
\end{proof}

\begin{remark} \label{remark 3.1}
As we have seen in the proof, the combinatorial Alexander dual of the dominance complex coincides with the simplicial complex consisting of the non-dominating sets in $G$. In \cite{BCS}, Brouwer, Csorba and Schrijver noted that the simplicial complex of non-dominating sets coincides with the neighborhood complex $\nbdcpx(\overline{G})$ of the complement of a graph $G$, but they did not consider the dominance complex. The combinatorial Alexander dual of the dominance complex was considered in Section 8 of \cite{EH}. In \cite{HT}, Heinrich and Tittman also noted a close relation between dominating sets and the neighborhood complex of the complement of a graph.
\end{remark}

\section{Applications}

In this section, we provide applications of Theorem \ref{main theorem}.

\subsection{Chromatic numbers of graph complements}

A \emph{proper $n$-coloring of a graph $G$} is a function $c \colon V(G) \to \{ 1, \cdots, n\}$ such that $\{ v, w\} \in E(G)$ implies $c(v) \ne c(w)$. The \emph{chromatic number\label{index:chromaticNumber} \(\chi(G)\) of a graph $G$} is the smallest integer $n$ such that there is a proper $n$-coloring of $G$. To determine the chromatic number is called the graph coloring problem, which is one of the most classical problems in graph theory.

As was mentioned in Section 1, Lov\'asz showed that there is a relation between the chromatic number of $G$ and the neighborhood complex of $G$.


\begin{theorem}[see Theorem 1.7 of Babson--Kozlov \cite{BK}] \label{theorem lovasz}
Let $G$ be a graph. Then
\[ \chi(G) \ge \conn_{\ZZ_2}(\nbdcpx(G)) + 2.\]
\end{theorem}

Now we are ready to prove Theorem \ref{theorem 2}.

\begin{proof}[Proof of Theorem \ref{theorem 2}]
Let $n$ be the number of vertices in $G$, and $d$ the homological dimension of $\domcpx(G) = \nbdcpx(\overline{G})^\vee$. By Lemma \ref{lemma conn hdim}, the $\ZZ_2$-homological connectivity $\conn_{\ZZ_2}(\nbdcpx(\overline{G}))$ of $\nbdcpx(\overline{G})$ is $n - d - 4$. Hence, by Theorem \ref{theorem lovasz}, we have
\[ \chi(\overline{G}) \ge \conn_{\ZZ_2}(\nbdcpx(\overline{G})) + 2 = n - d - 2 = n - \hdim_{\ZZ_2}(\domcpx(G)) - 2.\qedhere\]
\end{proof}

\subsection{Vertex cover number}

Recall that the \emph{independence number\label{index:independenceNumber} \(\alpha(G)\) of \(G\)} is the size of a maximum independent set of $G$. Since $C$ is a vertex cover if and only if $V(G) - C$ is an independent set, we have $\tau(G) + \alpha(G) = \# V(G)$. Here $\tau(G)$ denotes the vertex cover number (see Subsection 1.2).

In \cite{Matsushita6} the first author showed the following inequality which connects $\tau(G)$ with the homology of $\domcpx(G)$:

\begin{theorem}[Theorem 1 of \cite{Matsushita6}] \label{theorem Matsushita}
For a graph $G$,
\[ \conn_{\ZZ_2}(\domcpx(G)) + 2 \le \tau(G).\]
\end{theorem}

In this subsection, we give an alternative proof of Theorem \ref{theorem Matsushita}, using Theorem \ref{main theorem}. Before giving the proof, we need the following lemma. Recall that the \emph{clique number\label{index:cliqueNumber} \(\omega(G)\) of \(G\)} is the size of a maximum clique of $G$.

\begin{lemma} \label{lemma clique}
Let $m$ be the clique number \(\omega(G)\). Then
\[ \hdim_{\ZZ_2}(\nbdcpx(G)) \ge m - 2.\]
\end{lemma}

This lemma seems to be a folklore fact, but we provide a proof of this lemma for the reader's convenience.

In the proof of Lemma \ref{lemma clique}, we use the box complex of a graph \cite{Csorba}. Here we do not provide the precise definition of the box complex. We only need the properties of the box complex. We refer to \cite{Csorba} and \cite{MZ} for detailed studies of box complexes.

Let $\boxcpx(G)$ denote the box complex\label{index:boxComplex} of a graph $G$. The box complex is a free $\ZZ_2$-simplicial complex satisfying the following properties:
\begin{enumerate}[(a)]
\item The box complex $\boxcpx(G)$ and the neighborhood complex $\nbdcpx(G)$ are homotopy equivalent.

\item The box complex $\boxcpx(K_m)$ of the complete graph $K_m$ with $m$ vertices is $\ZZ_2$-homotopy equivalent to the $(m-2)$-sphere $S^{m-2}$. Here we consider $S^{m-2}$ as a $\ZZ_2$-space with the antipodal action.

\item If $G$ is a subcomplex of $H$, then $\boxcpx(G)$ is a $\ZZ_2$-subcomplex of $\boxcpx(H)$.
\end{enumerate}

We use the following simple lemma in $\ZZ_2$-homotopy theory.

\begin{lemma}[See Lemma 7 of \cite{Matsushita6} for example] \label{lemma coindex}
  Let $X$ be a free $\ZZ_2$-simplicial complex.
  Assume that there is a $\ZZ_2$-map from $S^k$ to $X$.
  Then \(\hdim_{\ZZ_2}(X)\ge k\).
\end{lemma}

Now we give proofs for Lemma \ref{lemma clique} and Theorem \ref{theorem Matsushita}.

\begin{proof}[Proof of Lemma \ref{lemma clique}]
Let $m$ be the clique number of $G$. Then properties (b) and (c) imply that there is a $\ZZ_2$-map from $S^{m-2}$ to $\boxcpx(G)$. Then Lemma \ref{lemma coindex} and property (a) imply that there is $k \ge m-2$ such that $\tilde{H}_k(\nbdcpx(G) ; \ZZ_2) \cong \tilde{H}_k(\boxcpx(G); \ZZ_2) \ne 0$.
\end{proof}

\begin{proof}[Proof of Theorem \ref{theorem Matsushita}]
Set $m = \omega(\overline{G}) = \alpha(G) = n - \tau (G)$. Then Theorem \ref{main theorem} implies
\[ \hdim_{\ZZ_2}(\domcpx(G)^\vee) = \hdim_{\ZZ_2} (\nbdcpx(\overline{G})) \ge m - 2 = n - \tau(G) - 2.\]
Then Lemma \ref{lemma conn hdim} completes the proof.
\end{proof}

\subsection{Dominance complexes of chordal graphs}
The homotopy types of the dominance complexes of forests were determined by Ehrenborg and Hetyei \cite{EH} (see also \cite{MT1}). Taylan generalized this result to chordal graphs. Recall that a graph $G$ is \emph{chordal} if it contains no induced $n$-cycle with $n \ge 4$. The homotopy types of the dominance complexes of chordal graphs were determined by Taylan as follows.

\begin{theorem}[Taylan \cite{Taylan}] \label{theorem chordal}
The dominance complex of a chordal graph \(G\) is homotopy equivalent to $S^{\tau(G) - 1}$.
\end{theorem}

In this subsection, we determine the homotopy type of $\nbdcpx(\overline{G})$ for a chordal graph $G$, and provide an alternative proof of Theorem \ref{theorem chordal} by Theorem \ref{main theorem}.

Before starting the proof, we recall some terminology of collapsing of simplicial complexes. For details, we refer to Subsection 6.4 of \cite{KozlovBook}. Let $X$ be a simplicial complex. A maximal simplex of $X$ is called a \emph{facet}. A face $\sigma$ of $X$ is said to be \emph{free} if $\sigma$ is not a facet and there is only one face containing $\sigma$. For a free face $\sigma$, let $X \setminus \sigma$ denote the subcomplex of $K$ whose simplices are the simplicies of $X$ not containing $\sigma$. A \emph{simplicial collapse} is to obtain $X \setminus \sigma$ from $X$ for a free face $\sigma$ in $X$. We say that a simplicial complex $X$ \emph{collapses} to a subcomplex $Y$ of $X$ if there is a sequence $X = X_0, X_1, \cdots, X_k = Y$ of subcomplexes of $X$ such that $X_i$ is a simplicial collapse of $X_{i-1}$.


In our situation, we need the fold lemma of the neighborhood complex. For a subset $S$ of $V(G)$, we write $G - S$ to indicate the induced subgraph of $G$ whose vertex set is $V(G) - S$. For a vertex $v \in V(G)$, we write $G - v$ instead of $G - \{ v\}$.

\begin{theorem} \label{theorem fold neighborhood}
Let $v$ and $w$ be vertices in a graph $G$. Assume that $N_G(v) \subset N_G(w)$. Then $\nbdcpx(G)$ collapses to $\nbdcpx(G - v)$.
\end{theorem}

This theorem seems to be a folklore fact. For example, this is deduced from Proposition 2.14 of \cite{BM}. Further generalizations of the fold lemma to Hom complexes have been studied by several authors (see \cite{BK, Dochtermann, Kozlov1, Matsushita1} for example).


We note the following lemma. The proof is straightforward and is omitted.

\begin{lemma} \label{lemma independence number}
Let $G$ be a graph, $v$ and $w$ distinct vertices in $G$ such that $N_G[v] \subset N_G[w]$. Then $\alpha(G - w) = \alpha(G)$.
\end{lemma}

To determine the homotopy type of \(\nbdcpx(\overline{G})\),
first we provide the following lemma.

\begin{lemma}
  \label{lemma chordal}
  Let \(G\) be a chordal graph with \(E(G) \neq \emptyset\).
  Then there is a non-isolated vertex \(w\in V(G)\) such that $\nbdcpx(\overline{G})$ collapses to $\nbdcpx(\overline{G - w})$ and $\alpha(G) = \alpha(G - w)$.
\end{lemma}

\begin{proof}
  By Dirac \cite{Dirac}, there is a simplicial vertex $v \in V(G)$, i.e., the subgraph induced by $N_G[v]$ is a clique.
  Since \(E(G) \neq \emptyset\),
  we can assume that $v$ is not isolated.
  Let $w \in N_G(v)$. Since $v$ is simplicial, we have $N_G[v] \subset N_G[w]$, which implies $N_{\overline{G}}(v) \supset N_{\overline{G}}(w)$.
  Hence Theorem \ref{theorem fold neighborhood} and Lemma \ref{lemma independence number} imply that $\nbdcpx(\overline{G})$ collapses to $\nbdcpx(\overline{G-w})$ and $\alpha(G) = \alpha(G-w)$, respectively.
\end{proof}

\begin{proposition} \label{proposition chordal}
Let $G$ be a chordal graph. Then $\nbdcpx(\overline{G})$ collapses to the neighborhood complex $\nbdcpx(K_{\alpha(G)})$ $(\simeq S^{\alpha(G) - 2})$ of the complete graph $K_{\alpha(G)}$.
\end{proposition}


\begin{proof}
  We prove the proposition by induction on \(\# E(G)\).
  If \(\# E(G) = 0\), we have \(\nbdcpx(\overline{G}) = \nbdcpx(K_n) = S^{n-2} = S^{\alpha (G) - 2}\) as desired.
  Assume \(\# E(G) > 0\).
  In this case,
  we can take \(w \in V(G)\) by Lemma \ref{lemma chordal}.
  Since \(\# E(G - w) < \# E(G)\), the induction hypothesis completes the proof.
  \end{proof}




To prove Theorem \ref{theorem chordal}, we use the following two propositions.

\begin{proposition} \label{proposition sphere dual}
Let $k$ and $n$ be non-negative integers such that $k \le n$. Let $\partial \Delta(n,k)$ be the simplicial complex whose underlying set is $\{ 0, 1, \cdots, n\}$ and whose set of simplices is
\[ \partial \Delta(n,k) = \{ \sigma \; | \; \textrm{$\sigma$ is a proper subset of $\{ 0,1, \cdots, k\}$}\}.\]
Then the combinatorial Alexander dual of $\partial \Delta(n,k)$ is homotopy equivalent to $S^{n-k-1}$.
\end{proposition}
\begin{proof}
If $k = 0$ or $n = k$, then the proof is obvious. Suppose that $0 < k < n$. Set
\[ X = \{ \sigma \subset \{ 0,1, \cdots, n\} \; | \; \{ k+1, \cdots, n\} \not\subset \sigma\}\]
and
\[ Y = \{ \sigma \subset \{ 0, 1, \cdots, n\} \; | \; \sigma \subset \{ k+1, \cdots, n\}\}.\]
Since $X$ is isomorphic to the join of $\Delta^{\{ 0, 1, \cdots, k\}}$ and $\partial \Delta(n, n-k-1)$, $X$ is contractible. It is clear that $Y$ is contractible. Since $X \cap Y$ is isomorphic to $\partial \Delta (n,n-k-1)$ and $X \cup Y = (\partial \Delta (n,k))^\vee$, it follows from the gluing lemma (see for example Remark 15.26 of \cite{KozlovBook}) that $(\partial \Delta(n,k))^\vee \simeq S^{n-k-1}$. This completes the proof.
\end{proof}

The following proposition is well known and easy to verify. See Proposition 8.2 of \cite{EH}, for example. Here note that in \cite{EH} Ehrenborg and Hetyei used the term ``vertex set'' to mean underlying set in our terminology.

\begin{proposition} \label{proposition collapse}
Let $X$ and $Y$ be simplicial complexes with the same underlying set. If $X$ collapses to $Y$, then $Y^\vee$ collapses to $X^\vee$.
\end{proposition}

Now we provide an alternative proof of Theorem \ref{theorem chordal}.

\begin{proof}[Proof of Theorem \ref{theorem chordal}]
By Proposition \ref{proposition chordal}, $\nbdcpx(\overline{G})$ collapses to $\nbdcpx(K_{\alpha(G)})$ with ghost vertices. Note that $\nbdcpx(K_{\alpha(G)})$ is isomorphic to $\partial \Delta(\# V(G) - 1, \alpha(G) - 1)$. Hence Proposition \ref{proposition sphere dual} implies that $\nbdcpx(K_{\alpha(G)})^\vee$ is homotopy equivalent to $S^{\# V(G) - \alpha(G) - 1} = S^{\tau(G) - 1}$. Then Proposition \ref{proposition collapse} implies that $\domcpx(G)$ and $S^{\tau(G) - 1}$ are homotopy equivalent.
%
\end{proof}

\section{Computational examples}

In this section, we provide a method to compute the homology group of the dominance complex. In \cite{Matsushita6}, the first author showed that the suspension $\Sigma (\domcpx(G)^\vee)$ is homotopy equivalent to the independence complex of a simple graph $G^\bt$ (Theorem \ref{theorem bowtie}). Using this theorem, several known results of dominance complexes can be reduced to known results on independence complexes of grid graphs. As a new result, we determine the homotopy types of $\domcpx(P_n \times P_2)$ of $P_n \times P_2$ (see Example \ref{example PnxP2} and Theorem \ref{theorem 1-connected})

In Subsection 5.1, we review the definitions and facts related to the independence complexes of finite simple graphs. In Subsection 5.2, we recall the definition and facts of $G^\bt$. In Subsection 5.3, we compute the homology groups of the dominance complexes. In Subsection 5.4, we obtain methods for determining if $\domcpx(G)$ is $1$-connected (Theorem \ref{theorem 1-connected criterion} and Corollary \ref{corollary easy criterion}). Using this, we determine the homotopy types of some of the dominance complexes discussed in Subsection 5.3 from their homology groups.

\subsection{Independence complexes}


Here we list several properties of independence complexes that follow immediately from the definition.

\begin{enumerate}[(1)]
\item If the graph $G$ is the disjoint union $G_1 \sqcup G_2$ of two subgraphs $G_1$ and $G_2$, then $\indepcpx(G)$ coincides with the join $\indepcpx(G_1) * \indepcpx(G_2)$ of $\indepcpx(G_1)$ and $\indepcpx(G_2)$.

\item In particular, if the graph $G$ has an isolated vertex, then $\indepcpx(G)$ is contractible.

\item Let $K_n$ be the complete graph with $n$ vertices. Then $\indepcpx(K_2)$ is $S^0$. Hence $\indepcpx(K_2 \sqcup G)$ coincides with the suspension $\Sigma \indepcpx(G)$ of $\indepcpx(G)$.

\item Let $H$ be a subgraph of $G$. Then $\indepcpx(H)$ is a subcomplex of $\indepcpx(G)$ if and only if $H$ is an induced subgraph of $G$.
\end{enumerate}

The independence complexes of finite simple graphs have been studied by many authors, and there are plenty of methods to determine their homotopy types (see for example \cite{Adamaszek1, Barmak, BMV, BLN, EH, Engstrom, GSS, Kozlov1}). Our method is the combination of the cofiber sequence in \cite{Adamaszek1} and the fold lemma \cite{Engstrom}. 

\begin{theorem}[see Proposition 3.1 of Adamaszek \cite{Adamaszek1}] \label{theorem cofiber}
The following sequence is a cofiber sequence:
\[ \indepcpx(G - N_G[v]) \hookrightarrow \indepcpx(G - v) \hookrightarrow \indepcpx(G). \]
In particular, if the inclusion $\indepcpx(G - N_G[v]) \hookrightarrow \indepcpx(G - v)$ is null-homotopic, then
\[ \indepcpx(G) \simeq \indepcpx(G - v) \vee \Sigma\indepcpx(G - N_G[v]).\]
\end{theorem}

\begin{theorem}[fold lemma of independence complex, Lemma 3.2 of \cite{Engstrom}]
  \label{theorem fold independence}
Let $v$ and $w$ be distinct vertices of $G$. If $N_G(v) \subset N_G(w)$, then the inclusion $\indepcpx(G - w) \hookrightarrow \indepcpx(G)$ is a homotopy equivalence.
\end{theorem}

Recently, Theorems \ref{theorem cofiber} and \ref{theorem fold independence} have been frequently used to compute the homotopy types of independence complexes. See \cite{BMV, Matsushita2.7, Matsushita5.5, MW1, MW2}, for example.

\subsection{Graph $G^\bt$}

In this subsection we recall the definition of the graph $G^\bt$ and its relation to the dominance complex (Theorem \ref{theorem bowtie}). Let $G$ be a graph. Define the graph $G^\bt$\label{index:bowtie} by
\[ V(G^\bt) = V(G) \times \{ +, -\},\quad E(G^\bt) = \{ \{ (v,+), (w,-)\} \; | \; w \in N_G[v] \}.\]
In the previous paper \cite{Matsushita6}, the first author showed the following, using Proposition 6.2 of Nagel and Reiner \cite{NR}.

\begin{theorem}[Corollary 4 of \cite{Matsushita6}] \label{theorem bowtie}
For a graph $G$, there is the following homotopy equivalence:
\[ \Sigma \big( \domcpx(G)^\vee \big) \simeq \indepcpx(G^\bt ).\]
\end{theorem}

\begin{remark}
Combining Theorems \ref{main theorem} and \ref{theorem bowtie}, we have $\indepcpx(G^\bt) \simeq \Sigma \nbdcpx(\overline{G})$. This homotopy equivalence is actually a reformulation of a known fact. Indeed, let $\boxcpx_0(G)$ be the simplicial complex defined by
\[ V(\boxcpx_0(G)) = V(G) \times \{ 1,2\}, \]
\[ \boxcpx_0(G) = \{ \sigma \uplus \tau \; | \; \textrm{$\sigma, \tau \subset V(G)$ such that every element of $\sigma$ and every element of $\tau$ are adjacent}\}.\]
Here $\sigma \uplus \tau$ means $\sigma \times \{ 1\} \cup \tau \times \{ 2\}$. It is known that $\boxcpx_0(G)$ and $\Sigma \nbdcpx(G)$ are homotopy equivalent (see Section 3 of \cite{MZ}). By the definitions of $G^\bt$ and the independence complex, it is straightforward to see that $\indepcpx(G^\bt)$ and $\boxcpx_0(\overline{G})$ are isomorphic. Hence we have $\indepcpx(G^\bt) \simeq \Sigma \nbdcpx(\overline{G})$. By using Theorem \ref{main theorem}, we have proved Theorem \ref{theorem bowtie}. This is another application of Theorem \ref{main theorem}.

\end{remark}



We will frequently use the next lemma (Lemma \ref{lemma bipartite bowtie}) in the subsequent sections. To state it precisely, we now explain several terms. The \emph{cartesian product of graphs $G$ and $H$} is the graph defined by
\[ V(G \times H) = V(G) \times V(H),\]
\[ E(G \times H) = \{ \{ (v,w),(v',w')\} \; | \; \textrm{($v = v'$ and $w \sim_H w'$) or ($v \sim_G v'$ and $w = w'$)}\}.\]
We define the \emph{path graph}\label{def path graph} \(P_n\) consisting of $n$ vertices by
\[ V(P_n) = \{ 1, 2, \cdots, n\}, \quad E(P_n) = \{ x,y \mid x,y \in V(P_n), \; |x-y| = 1\}.\]

\begin{lemma} \label{lemma bipartite bowtie}
  \label{lemma bowtie}
  Let \(G\) be a bipartite graph.
  Then the graph \(G^\bt\) is isomorphic to the cartesian product \(G\times P_2\).
\end{lemma}
\begin{proof}
Let $\varepsilon \colon G \to K_2$ be a $2$-coloring. Define the graph homomorphism $f \colon G^\bt \to G \times K_2$ by
\[ f((v, +)) = \begin{cases}
(v, 1) & (\varepsilon (v) = 1) \\
(v, 2) & (\varepsilon (v) = 2)
\end{cases} \quad
\textrm{and} \quad
f((v, -)) = \begin{cases}
(v,2) & (\varepsilon (v) = 1) \\
(v,1) & (\varepsilon (v) = 2).
\end{cases}\]
Define a graph homomorphism $g \colon G \times K_2 \to G^\bt$ by
\[ g(v,1) = \begin{cases}
(v, +) & (\varepsilon (v) = 1) \\
(v, -) & (\varepsilon (v) = 2)
\end{cases}
\quad
\textrm{and}
\quad
g(v,2) = \begin{cases}
(v, -) & (\varepsilon (v) = 1) \\
(v, +) & (\varepsilon (v) = 2).
\end{cases}
\]
Then $g$ is the inverse of $f$.
\end{proof}

\subsection{Computational examples}

In this subsection, we see that several results of homology groups of dominance complexes can be deduced from some computations of the independence complex of grid graphs.
Here we freely use propositions in Subsection \ref{subsection topological facts} without mention.

Define the \emph{cycle graph}\label{def cycle graph} \(C_n\) by
\[ V(C_n) = \ZZ / n \ZZ, \quad E(C_n) = \{ \{x, y\} \mid x, y\in V(C_n),\; x-y=\pm 1 \}.\]
The independence complexes of $P_n \times P_k$ and $C_n \times P_k$ have been studied by several authors (see \cite{Adamaszek2, BLN, Iriye, Jonsson1, Jonsson2, Jonsson3, Kozlov, MW1, MW2, Thapper}). Here we use these results to determine the homotopy types of dominance complexes of the path graphs $P_n$ and the cycle graphs $C_n$. We should state that the computations of the homology groups for the dominance complexes given here are all known with the exception of Example \ref{example PnxP2}.

\begin{example}[$P_n$]
Here we consider the dominance complex $\domcpx(P_n)$ of the path graph $P_n$ consisting of $n$ vertices. By Lemma \ref{lemma bowtie}, $P_n^\bt$ is isomorphic to $P_n \times P_2$. The homotopy type of $\indepcpx(P_n \times P_2)$ is determined by Adamaszek \cite{Adamaszek2} as follows:
\[ \indepcpx(P_n \times P_2) \simeq \begin{cases}
	S^{k-1} & (n = 2k) \\
	S^k & (n = 2k +1).
\end{cases}\]
Hence we have
\[ \Sigma \big( \domcpx(P_n)^\vee \big) \simeq \indepcpx(P_n^\bt) \simeq \indepcpx(P_n \times P_2) \simeq \begin{cases}
S^{k-1} & n = 2k \\
S^k & n = 2k + 1.
\end{cases}\]
By Theorem \ref{theorem alexander}, we have
\[ \tilde{H}_i (\domcpx(P_{2k})) \cong \tilde{H}_i(\domcpx(P_{2k + 1})) \cong \begin{cases}
\ZZ & (i = k - 1) \\
0 & \textrm{(otherwise).}
\end{cases} \]
Note that this result is due to Ehrenborg and Hetyei \cite{EH}.
\end{example}

\begin{example}[$C_{2k}$] \label{example even C}
Consider the cycle graph $C_{2k}$ consisting of $2k$ vertices. By Lemma \ref{lemma bowtie}, $C_{2k}^\bt$ is isomorphic to $C_{2k} \times P_2$. The homotopy types of independence complexes of $C_n \times P_2$ were determined by Thapper \cite{Thapper} as follows:
\[ \indepcpx(C_n \times P_2) \simeq \begin{cases}
\bigvee_3 S^{2k - 1} & (n = 4k) \\
S^{2k-2 + i} & (n = 4k + i, i = 1,2,3).
\end{cases}\]
Hence Theorems \ref{main theorem} and \ref{theorem alexander} imply that
\[ \tilde{H}_i(\domcpx(C_{4k})) \cong \begin{cases}
\ZZ^3 & (i = 2k - 1)\\
0 & \textrm{(otherwise)}
\end{cases}
\quad
\textrm{and}
\quad
\tilde{H}_i(\domcpx(C_{4k + 2})) \cong \begin{cases}
\ZZ & (i = 2k) \\
0 & \textrm{(otherwise).}
\end{cases}\]
This result is essentially due to Taylan (see Theorem 5.2 of \cite{Taylan}). Indeed, the dominance complex $\domcpx(C_n)$ of the $n$-cycle $C_n$ is the $P_3$-devoid complex of $C_n$ discussed in \cite{Taylan}. Taylan actually determined the homotopy types of $\domcpx(C_n)$, and we will also see it in Theorem \ref{theorem cycle homotopy equivalence}.
\end{example}

\begin{example}[$C_{2k+1}$] \label{example odd C}
$C_{2k+1}^\bt$ is isomorphic to $M_{2, 2k+1}$. Here $M_{2, n}$ is the graph obtained from $P_{n+1} \times P_2$ by identifying $(1,1)$ with $(n+1,2)$ and $(1,2)$ with $(n+1,1)$. Okura \cite{Okura} determined the homotopy type of $\indepcpx(M_{2, 2k+1})$ as follows:
\[ \indepcpx(M_{2, n}) \simeq \begin{cases}
S^{2k-1} & (n = 4k) \\
S^{2k} & (n = 4k + 1) \\
\bigvee_3 S^{2k} & (n = 4k + 2) \\
S^{2k} & (n = 4k + 3).
\end{cases}\]
Hence Theorem \ref{main theorem} and Theorem \ref{theorem alexander} imply that
\[ \tilde{H}_i (\domcpx(C_{4k + 1})) \cong \begin{cases}
\ZZ & (i = 2k - 1) \\
0 & \textrm{(otherwise)}
\end{cases}
\quad
\textrm{and}
\quad
\tilde{H}_i (\domcpx(C_{4k + 3})) \cong \begin{cases}
\ZZ & (i = 2k + 1) \\
0 & \textrm{(otherwise)}
\end{cases}\]
This result is due to Taylan (see Theorem 5.2 of \cite{Taylan}). We will also see that $\domcpx(C_n)$ is homotopy equivalent to a wedge of spheres (see Theorem \ref{theorem cycle homotopy equivalence}).
\end{example}

We conclude this section by determining the homology group of $\domcpx(P_n \times P_2)$ from $\indepcpx(P_n \times C_4)$. This is a new example of computing the homology groups of dominance complexes.

\begin{example}[$P_n \times P_2$] \label{example PnxP2}
By Lemma \ref{lemma bowtie}, we have $(P_n \times P_2)^\bt \cong P_n\times P_2\times P_2 \cong P_n \times C_4$. The homotopy type of $\indepcpx(P_n \times C_4)$ was determined by Thapper \cite{Thapper} as follows:
\[ \indepcpx(P_n \times C_4) \simeq \begin{cases}
\bigvee_{2k+1} S^{2k - 1} & (n = 2k) \\
\bigvee_{2k} S^{2k} & (n = 2k + 1).
\end{cases}\]
Hence Theorem \ref{main theorem} and Theorem \ref{theorem alexander} imply that
\[ \tilde{H}_i(\domcpx(P_{2k} \times P_2)) \cong \begin{cases}
\ZZ^{2k + 1} & (i = 2k + 1)\\
0 & \textrm{(otherwise)}
\end{cases}
\quad
\textrm{and}
\quad
\tilde{H}_i(\domcpx(P_{2k + 1} \times P_2)) \cong \begin{cases}
\ZZ^{2k} & (i = 2k) \\
0 & \textrm{(otherwise).}
\end{cases}
\]
In fact, we will see that $\domcpx(P_n \times P_2)$ are homotopy equivalent to a wedge of spheres (see Theorem \ref{theorem grid homotopy equivalence}).
\end{example}

\subsection{1-connectedness of the dominance complex}

The main object of this paper is the homology groups of the dominance complex. In general, the homology groups do not determine the homotopy type. In fact, there are many spaces with the same homology groups whose homotopy types are different. However, under certain circumstances, the homology groups sometimes determine the homotopy type as the following well-known theorem shows:

\begin{theorem}[see Example 4.34 of \cite{Hatcher}] \label{theorem homology to homotopy}
Let $X$ be a $1$-connected simplicial complex, and $n$ a positive integer. Assume that $\tilde{H}_i(X ; \ZZ)$ is zero if $i \ne n$, and that $\tilde{H}_n(X ; \ZZ)$ is the finitely generated free abelian group $\ZZ^r$. Then there is the following homotopy equivalence:
\[ X \simeq \bigvee_{r} S^n.\]
\end{theorem}

Here, recall that a topological space $X$ is $k$-connected if $\pi_i(X)$ is trivial for every $i \le k$. The connectivity of $X$ is the supremum of $k$ such that $X$ is $k$-connected. We refer the reader to \cite{Hatcher} for details.

All of the examples we provided in Subsection 5.2 have homology groups of the form described in Theorem \ref{theorem homology to homotopy}. Therefore, in this subsection, we consider a method for determining the $1$-connectedness of the dominance complex. The main result of this subsection is Theorem \ref{theorem 1-connected criterion}. We also show a relation between the minimum degree and the connectivity of the dominance complex (Proposition \ref{proposition connectivity}).

It turns out that the $1$-connectedness of the dominance complex is closely related to the degrees of vertices. First, consider the case that a graph $G$ has a vertex $v$ of degree $0$, i.e., $v$ is an isolated vertex. Then it is clear that $\domcpx(G) = \domcpx(G - v)$, and hence we only consider that the minimum degree of $G$ is at least $1$. If $v$ is a vertex of degree $1$, then the following is actually known to hold:

\begin{theorem}[see Lemma 4.10 of Marietti--Testa \cite{MT1}]
Let $G$ be a graph and $v$ a vertex in $G$ whose degree is at most $1$. Then there is the following homotopy equivalence:
\[ \domcpx(G) \simeq \Sigma \domcpx(G - N_G[v]).\]
\end{theorem}

Thus it suffices to consider the case that the minimum degree of $G$ is at least $2$. Now we are ready to state our main result in this subsection:

\begin{figure}[t]



\tikzset{my dot/.style={circle, fill=black, minimum size=5pt, inner sep=0}}
\begin{tikzpicture}[scale=1.5]
  \draw[fill=lightgray]
    (0,0) node[my dot, label=below:{\(v_0\)}] (v0) {} --
    (1,1.7) node[my dot, label=above:{\(v_1\)}] (v1) {} --
    (2,0) node[my dot, label=below:{\(v_2\)}] (v2) {} --
    cycle;
  \node[my dot, label=below:{\(a\)}] (a) at (1, 0.8) {};
  \draw (a) -- (v0);
  \draw (a) -- (v1);
  \draw (a) -- (v2);
  \draw
    (v2) --
    (3,0.5) node[my dot, label=below:{\(v_3\)}] {} --
    (4,0.2) node[my dot, label=below:{\(v_4\)}] {};
\end{tikzpicture}
\caption{$(v_0, v_1, v_2) \simeq (v_0, v_2)$} \label{figure homotopy}
\end{figure}

\begin{theorem} \label{theorem 1-connected criterion}
Let $G$ be a non-empty graph whose minimum degree is at least $2$. Assume that for every vertex $v$ of $G$ whose degree is $2$ there is a vertex $a$ satisfying the following conditions:
\begin{enumerate}[$(1)$]
\item $a \not\in N_G[v]$.

\item $N_G(v) \ne N_G(a)$.

\item If the degree of $u \in N_G(v)$ is $2$, then $a \not\in N_G(u)$.
\end{enumerate}
Then the dominance complex $\domcpx(G)$ of $G$ is $1$-connected.
\end{theorem}

The proof of the theorem is based on the edge-path group described below. For details on edge-path groups, see Spanier \cite{Spanier}.

Let $K$ be a simplicial complex and $v$ a vertex of $K$. An \emph{edge-path} of $K$ is a sequence of vertices $\gamma = (v_0, v_1, \cdots, v_n)$ such that $\{ v_{i-1}, v_i\}$ is a simplex of $K$ for $i = 1, \cdots ,n$. Note that \(\{v_{i-1}, v_i\}\) denotes a 0-simplex if \(v_{i-1} = v_i\). If $v_0 = v_n = v$, we say that the edge-path $\gamma$ is an edge-loop of a based simplicial complex $(K,v)$.

Let $E(K,v)$ denote the set of edge-loops of $(K,v)$. Consider the following two transformations:
\begin{enumerate}[(a)]
\item If $n \ge 2$ and $\{ v_{i-1}, v_i, v_{i+1} \} \in K$, then $(v_0, \cdots, v_n) \simeq (v_0, \cdots, v_{i-1}, v_{i+1}, \cdots, v_n)$.

\item If $n\ge 1$ and $v_i = v_{i+1}$, then $(v_0, \cdots, v_n) \simeq (v_0, \cdots, v_{i-1}, v_{i+1}, \cdots, v_n)$.
\end{enumerate}
Let $\simeq$ denote the equivalence relations generated by (a) and (b), and set $\mathcal{E}(K,v) = E(K,v) / \simeq$. Two edge-loops $\gamma$ and $\gamma'$ are \emph{homotopic} if $\gamma \simeq \gamma'$. Then $\mathcal{E}(K,v)$ becomes a group by the concatenation of paths, and we call $\mathcal{E}(K,v)$ the \emph{edge-path group of $(K,v)$}. The following classical theorem is fundamental:

\begin{theorem}[see Corollary 3.6.17 of \cite{Spanier}]
$\mathcal{E}(K,v) \cong \pi_1(K,v)$
\end{theorem}

Thus, to show that $K$ is 1-connected, it suffices to show that $K$ is path-connected and $\mathcal{E}(K,v)$ is trivial.

\begin{proof}[Proof of Theorem \ref{theorem 1-connected criterion}]
Since the minimum degree of $G$ is at least $2$, every $2$-element subset of $V(G)$ is a simplex of $\domcpx(G)$. This implies that $\domcpx(G)$ is path-connected. In the rest of the proof, we show that the edge-path group of $\domcpx(G)$ is trivial. To see this, we show that an edge-loop of $\domcpx(G)$ of length at least $2$ is homotopic to a shorter edge-loop. Let $\gamma = (v_0, v_1, v_2, \cdots, v_n)$ be an edge-loop and assume that the length of edge-loops homotopic to $\gamma$ is at least $n$. We would like to derive the contradiction when $n \ge 2$.

Assume $n \ge 2$. Then $v_0, v_1, v_2$ are distinct vertices. If $\{ v_0, v_1, v_2\} \in \domcpx(G)$, then $\gamma \simeq (v_0, v_2, v_3, \cdots, v_n)$, and hence this contradicts to the assumption on $\gamma$. Hence $\{ v_0, v_1, v_2\}$ is a non-face of $\domcpx(G)$.

Since $\{ v_0, v_1, v_2\}$ is a non-face of $\domcpx(G)$, there is $v \in V(G)$ such that $N_G[v] \subset \{ v_0, v_1, v_2\}$. Since the minimum degree of $G$ is at least $2$, we have $N_G[v] = \{ v_0, v_1, v_2\}$. Set $\{ u,w\} = N_G(v)$. Let $a$ be a vertex of $G$ satisfying (1), (2) and (3) in the assertion of this theorem. We would like to show the following claim:

\smallskip
\noindent
{\bf Claim.} The union of any $2$-element subset of $\{ v_0, v_1, v_2\} = \{ u,v,w\}$ and $\{ a\}$ is a simplex of $\domcpx(G)$.

\smallskip
We now show Claim. Suppose that $\{ u,w,a\}$ is a non-face of $\domcpx(G)$. Since the minimum degree of $G$ is at least $2$, there is $x \in V(G)$ such that $\{ u,w,a\} = N_G[x]$. Since $v \not\in \{ u,w, a\}$, we have $x \ne u,w$. Hence we have $x = a$ and $N_G(a) = \{ u,w\} = N_G(v)$. This contradicts the assumption (2).

Next suppose that $\{ v,w, a\}$ is not a simplex of $\domcpx(G)$. Then, since the minimum degree of $G$ is at least $2$, there is $x \in V(G)$ such that $N_G[x] = \{ v,w,a\}$. Since $N_G[x] \ne N_G[v]$, we have $x \ne v$. Since $v \in N_G[x]$, we have $x = w$. Hence the degree of $w$ is $2$ and $a \in N_G[w]$. This contradicts the assumption (3). Thus $\{ v,w,a\}$ is a simplex of $\domcpx(G)$. In a similar way, we can show that $\{ u,v,a\}$ is a simplex of $\domcpx(G)$. This completes the proof of Claim.

Finally, by Claim and Figure \ref{figure homotopy},
we can see that
\begin{align}
  \gamma &= (v_0, v_1, v_2, v_3, \cdots, v_n) \simeq (v_0, a, v_1, a, v_2, v_3 \cdots, v_n) \\
  &\simeq (v_0, a, v_2, v_3 \cdots, v_n) \simeq (v_0, v_2, v_3, \cdots, v_n).
\end{align}
This contradicts the assumption of $\gamma$. This completes the proof.
\end{proof}

The following corollary to Theorem \ref{theorem 1-connected criterion} is useful.

\begin{corollary} \label{corollary easy criterion}
Let $G$ be a graph whose minimum degree of $G$ is at least $2$. Assume that for every vertex $v$ whose degree is $2$ there is a vertex $a$ such that the length of the shortest path between $v$ and $a$ is at least $3$. Then $\domcpx(G)$ is $1$-connected.
\end{corollary}

It follows from Theorem \ref{theorem 1-connected criterion} that if the minimum degree of $G$ is at least $3$, then $\domcpx(G)$ is $1$-connected. However, this is easily generalized as follows:

\begin{proposition} \label{proposition connectivity}
Let $k$ be the minimum degree of $G$. Then $\domcpx(G)$ is $(k-2)$-connected.
\end{proposition}
\begin{proof}
By Lemma \ref{lemma free}, every subset of $V(G)$ whose cardinality is at most $k$ is a simplex of $\domcpx(G)$. This means that the $(k-1)$-skeleton of $\domcpx(G)$ is coincides with the $(k-1)$-skeleton of $\Delta^{V(G)}$. Here $\Delta^{V(G)}$ denotes the simplex whose vertex set is $V(G)$. Then for every integer $i$ with $0 \le i \le k-2$, we have
\[ \pi_i(\domcpx(G)) \cong \pi_i(\domcpx(G)_{k-1}) \cong \pi_i(\Delta^{V(G)}_{k-1}) \cong \pi_i(\Delta^{V(G)}) \cong *\]
Here $X_{k-1}$ denotes the $(k-1)$-skeleton of a simplicial complex $X$, and the first and third isomorphisms are deduced from Corollary 4.12 of \cite{Hatcher}. This completes the proof.
\end{proof}

Using Corollary \ref{corollary easy criterion} we have the following. Note that (1) of the following was obtained by Taylan \cite{Taylan}.

\begin{theorem} \label{theorem 1-connected}
The following hold:
\begin{enumerate}[$(1)$]
\item The dominance complex $\domcpx(C_n)$ of $C_n$ is homotopy equivalent to a wedge of spheres.

\item The dominance complex $\domcpx(P_n \times P_2)$  of $P_n \times P_2$ is homotopy equivalent to a wedge of spheres.
\end{enumerate}
\end{theorem}
\begin{proof}
It follows from Corollary \ref{corollary easy criterion} that $\domcpx(C_n)$ is $1$-connected if $n \ge 6$. In the case of $P_n \times P_2$, it follows from Corollary \ref{corollary easy criterion} that $\domcpx(P_n \times P_2)$ is $1$-connected if $n \ge 3$. Hence Theorem \ref{theorem homology to homotopy} and Examples \ref{example even C}, \ref{example odd C}, and \ref{example PnxP2} imply that these complexes are wedge of spheres.
In the remaining cases, i.e., $\domcpx(C_3)$, $\domcpx(C_4) \cong \domcpx(P_2 \times P_2)$ and $\domcpx(C_5)$, it is not difficult to show directly
\[ \domcpx(C_3) = S^1, \quad \domcpx(C_4) \simeq S^1 \vee S^1 \vee S^1, \quad \domcpx(C_5) \simeq S^1.\]
This completes the proof.
\end{proof}

Note that the homology groups of these complexes were determined by Examples \ref{example even C}, \ref{example odd C} and \ref{example PnxP2}. Thus Theorem \ref{theorem 1-connected} determines the homotopy types of these complexes.

\begin{theorem}[Taylan \cite{Taylan}] \label{theorem cycle homotopy equivalence}
There are the following homotopy equivalences:
\[ \domcpx(C_{4k+i}) \simeq \begin{cases}
\bigvee_3 S^{2k-1} & (i = 0) \\
S^{2k-2 + i} & (i=1,2,3).
\end{cases}\]
\end{theorem}

Example \ref{example PnxP2} and Theorem \ref{theorem 1-connected} show the following new result.

\begin{theorem}\label{theorem grid homotopy equivalence}
The following hold:
\[ \domcpx(P_{2k} \times P_2) \simeq \bigvee_{2k+1} S^{2k+1} \quad \textrm{and} \quad \domcpx(P_{2k+1} \times P_2) \simeq \bigvee_{2k} S^{2k}.\]
\end{theorem}



\begin{figure}[t]
\begin{picture}(250,140)(0,-20)
\multiput(30,0)(27,18){4}{\circle*{3}}
\multiput(70,0)(27,18){4}{\circle*{3}}
\multiput(110,0)(27,18){4}{\circle*{3}}
\multiput(30,40)(27,18){4}{\circle*{3}}
\multiput(70,40)(27,18){4}{\circle*{3}}
\multiput(110,40)(27,18){4}{\circle*{3}}

\multiput(30,0)(27,18){4}{\line(0,1){40}}
\multiput(30,0)(27,18){4}{\line(1,0){80}}
\multiput(30,40)(27,18){4}{\line(1,0){80}}

\multiput(70,0)(27,18){4}{\line(0,1){40}}
\multiput(110,0)(27,18){4}{\line(0,1){40}}

\multiput(30,0)(40,0){3}{\line(3,2){95}}
\multiput(30,40)(40,0){3}{\line(3,2){95}}

\put(63,45){$v_n$}

\put(80,-23){$\Gamma_n = P_n \times P_3 \times P_2$}
\end{picture}
\caption{Graph $\Gamma_n$} \label{figure Gamma}
\end{figure}

\begin{figure}[b]
\begin{picture}(380,140)(0,-20)
\multiput(30,40)(27,18){3}{\circle*{3}}
\multiput(70,0)(27,18){3}{\circle*{3}}
\multiput(110,40)(27,18){3}{\circle*{3}}

\multiput(57, 18)(27,18){2}{\circle*{3}}
\multiput(97, 58)(27,18){2}{\circle*{3}}
\multiput(137, 18)(27,18){2}{\circle*{3}}

\multiput(30,40)(80,0){2}{\line(3,2){65}}
\put(70,0){\line(3,2){65}}

\multiput(57,18)(40,0){3}{\line(0,1){40}}
\multiput(84,36)(40,0){3}{\line(0,1){40}}
\multiput(57,18)(27,18){2}{\line(1,0){80}}
\multiput(57,58)(27,18){2}{\line(1,0){80}}
\multiput(57,18)(80,0){2}{\line(3,2){38}}
\put(97,58){\line(3,2){38}}

\put(80,-23){$X_n$}

\multiput(200,0)(27,18){3}{\circle*{3}}
\multiput(200,40)(27,18){3}{\circle*{3}}
\multiput(280,0)(27,18){3}{\circle*{3}}
\multiput(280,40)(27,18){3}{\circle*{3}}
\multiput(267, 18)(27,18){2}{\circle*{3}}
\multiput(267, 58)(27,18){2}{\circle*{3}}

\multiput(200,0)(0,40){2}{\line(3,2){65}}
\multiput(280,0)(0,40){2}{\line(3,2){65}}
\multiput(200,0)(27,18){3}{\line(0,1){40}}
\multiput(280,0)(27,18){3}{\line(0,1){40}}

\multiput(227,18)(27,18){2}{\line(1,0){80}}
\multiput(227,58)(27,18){2}{\line(1,0){80}}
\multiput(267,18)(0,40){2}{\line(3,2){38}}
\multiput(267,18)(27,18){2}{\line(0,1){40}}

\put(250,-23){$Y_n$}

\end{picture}
\caption{Graphs $X_n$ and $Y_n$} \label{figure XY}
\end{figure}

\section{The dominance complex of $P_n \times P_3$}
In this section, we determine the homology groups of the dominance complex $\domcpx(P_n \times P_3)$ of the cartesian product of $P_n$ and $P_3$.
By the argument mentioned in Section 5, it suffices to determine the homotopy types of the independence complex of $(P_n \times P_3)^\bt \cong P_n \times P_3 \times P_2$ (see Lemma \ref{lemma bowtie}). Figure \ref{figure Gamma} depicts the graph $P_n \times P_3 \times P_2$.

Set $\Gamma_n = P_n \times P_3 \times P_2$. In our proof, it is important to study the independence complexes of the following induced subgraphs $X_n$, $Y_n$, $A_n$ and $B_n$ of $\Gamma_n$.

\begin{definition} \label{definition XYAB}
Define $X_n$, $Y_n$, $A_n$ and $B_n$ by
\[ X_n = \Gamma_n - \{ (n,1,1), (n,2,2), (n,3,1) \},\]
\[ Y_n = \Gamma_n - \{ (n,2,1), (n,2,2)\},\]
\[ A_n = \Gamma_n - \{ (n,1,1), (n,3,1)\},\]
\[ B_n = \Gamma_n - \{ (n,2,1) \}.\]
Set $v_n = (n,2,2)$. See Figures \ref{figure XY} and \ref{figure AB}.
\end{definition}

\subsection{Homotopy types of \(\indepcpx(X_n)\) and \(\indepcpx(Y_n)\)}

We first determine the homotopy types of $\indepcpx(X_n)$ and $\indepcpx(Y_n)$. Suppose that $n \ge 3$. The fold lemma (Theorem \ref{theorem fold independence}) implies that
\[ \indepcpx(X_n) \simeq \Sigma^3 \indepcpx(X_{n-2}) \quad \textrm{and} \quad \indepcpx(Y_n) \simeq \Sigma^3 \indepcpx(Y_{n-2}).\]
See Figures \ref{figure X} and \ref{figure Y}. Hence the homotopy types of $\indepcpx(X_n)$ and $\indepcpx(Y_n)$ are recursively determined by $\indepcpx(X_1)$, $\indepcpx(X_2)$, $\indepcpx(Y_1)$ and $\indepcpx(Y_2)$. The homotopy types of these complexes are determined by the fold lemma, and we have
\[ \indepcpx(X_1) \simeq \pt, \quad \indepcpx(X_2) \simeq S^2, \quad \indepcpx(Y_1) \simeq S^1, \quad \indepcpx(Y_2) \simeq  S^2.\]
In summary, the homotopy types of $\indepcpx(X_n)$ and $\indepcpx(Y_n)$ are as follows.

\begin{lemma} \label{lemma X}
For $n \ge 1$, there is the following homotopy equivalence:
\[ \indepcpx(X_n) \simeq \begin{cases}
S^{3k-1} & n = 2k \\
\pt & n = 2k + 1.
\end{cases}\]
\end{lemma}

\begin{lemma} \label{lemma Y}
For $n \ge 1$, there is the following homotopy equivalence:
\[ \indepcpx(Y_n) \simeq \begin{cases}
S^{3k-1} & n = 2k \\
S^{3k+1} & n = 2k + 1.
\end{cases} \]
\end{lemma}

\subsection{Homotopy types of \(\indepcpx(A_n)\) and \(\indepcpx(B_n)\)}

We next consider the homotopy types of $\indepcpx(A_n)$ and $\indepcpx(B_n)$. Recall that $v_n = (n,2,2)$ (see Definition \ref{definition XYAB}).

\begin{figure}[t]
\begin{picture}(480,140)(20,-20)
\multiput(57,18)(27,18){3}{\circle*{3}}
\multiput(70,0)(27,18){4}{\circle*{3}}
\multiput(137,18)(27,18){3}{\circle*{3}}
\multiput(30,40)(27,18){4}{\circle*{3}}
\multiput(97,58)(27,18){3}{\circle*{3}}
\multiput(110,40)(27,18){4}{\circle*{3}}

\multiput(30,40)(80,0){2}{\line(3,2){95}}
\put(70,0){\line(3,2){95}}
\multiput(57,18)(27,18){3}{\line(1,0){80}}
\multiput(57,58)(27,18){3}{\line(1,0){80}}
\multiput(57, 18)(27,18){3}{\line(0,1){40}}
\multiput(97, 18)(27,18){3}{\line(0,1){40}}
\multiput(137, 18)(27,18){3}{\line(0,1){40}}

\multiput(57,18)(80,0){2}{\line(3,2){68}}
\put(97,58){\line(3,2){68}}

\put(80,-23){$X_n$}


\multiput(290,0)(27,18){2}{\circle*{3}}
\multiput(250,40)(27,18){2}{\circle*{3}}
\multiput(330,40)(27,18){2}{\circle*{3}}
\multiput(277, 18)(80,0){2}{\circle{3}}
\put(317,58){\circle{3}}
\multiput(304,76)(80,0){2}{\circle{3}}

\put(344,36){\circle{3}}
\multiput(304,36)(80,0){2}{\circle*{3}}
\multiput(331, 54)(40,0){3}{\circle*{3}}
\multiput(331,94)(40,0){3}{\circle*{3}}
\put(344,76){\circle*{3}}

\multiput(304,36)(80,0){2}{\line(3,2){38}}
\put(344,76){\line(3,2){38}}
\multiput(331,54)(40,0){3}{\line(0,1){40}}
\multiput(331,54)(0,40){2}{\line(1,0){80}}
\multiput(331,94)(80,0){2}{\line(3,2){11}}
\put(371,54){\line(3,2){11}}

\multiput(250,40)(80,0){2}{\line(3,2){27}}
\put(290,0){\line(3,2){27}}

\put(280,-23){$X_{n-2} \sqcup K_2 \sqcup K_2 \sqcup K_2$}
\end{picture}
\caption{$\indepcpx(X_n) \simeq \Sigma^3 \indepcpx(X_{n-2})$} \label{figure X}
\end{figure}

\begin{figure}[b]
\begin{picture}(480,140)(20,-20)
\multiput(30,0)(27,18){4}{\circle*{3}}
\multiput(30,40)(27,18){4}{\circle*{3}}
\multiput(110,0)(27,18){4}{\circle*{3}}
\multiput(110,40)(27,18){4}{\circle*{3}}
\multiput(97,18)(27,18){3}{\circle*{3}}
\multiput(97,58)(27,18){3}{\circle*{3}}

\multiput(30,0)(0,40){2}{\line(3,2){95}}
\multiput(110,0)(0,40){2}{\line(3,2){95}}
\multiput(30,0)(27,18){4}{\line(0,1){40}}
\multiput(57,18)(27,18){3}{\line(1,0){80}}
\multiput(57,58)(27,18){3}{\line(1,0){80}}
\multiput(97,18)(0,40){2}{\line(3,2){68}}
\multiput(97,18)(27,18){3}{\line(0,1){40}}
\multiput(110,0)(27,18){4}{\line(0,1){40}}

\put(80,-23){$Y_n$}


\multiput(250,0)(0,40){2}{\circle*{3}}
\multiput(330,0)(0,40){2}{\circle*{3}}
\multiput(277,18)(0,40){2}{\circle{3}}
\multiput(357,18)(0,40){2}{\circle{3}}
\multiput(317,18)(0,40){2}{\circle*{3}}

\multiput(304,36)(0,40){2}{\circle*{3}}
\multiput(384,36)(0,40){2}{\circle*{3}}
\multiput(331,54)(40,0){3}{\circle*{3}}
\multiput(331,94)(40,0){3}{\circle*{3}}

\multiput(344,36)(0,40){2}{\circle{3}}

\multiput(304,36)(80,0){2}{\line(0,1){40}}
\multiput(331,54)(0,40){2}{\line(1,0){80}}
\multiput(331,54)(40,0){3}{\line(0,1){40}}

\multiput(250,0)(80,0){2}{\line(0,1){40}}
\put(317,18){\line(0,1){40}}

\multiput(304, 36)(80,0){2}{\line(3,2){38}}

\multiput(304,36)(80,0){2}{\line(3,2){38}}
\multiput(304,76)(80,0){2}{\line(3,2){38}}
\multiput(371,54)(0,40){2}{\line(3,2){11}}

\put(280,-23){$Y_{n-2} \sqcup K_2 \sqcup K_2 \sqcup K_2$}
\end{picture}
\caption{$\indepcpx(Y_n) \simeq \Sigma^3 \indepcpx(Y_{n-2})$} \label{figure Y}
\end{figure}

\begin{lemma} \label{lemma AB}
The following hold:
\begin{enumerate}[$(1)$]
\item For $n \ge 1$, the following hold:
\[ A_n - v_n = X_n, \quad B_n - v_n = Y_n.\]

\item For $n \ge 2$, the following holds:
\[ A_n - N[v_n] \cong B_{n-1}.\]

\item For $n \ge 3$, the following holds:
\[ \indepcpx(B_n - N[v_n]) \simeq \Sigma^2 \indepcpx(A_{n-2}).\]
\end{enumerate}
\end{lemma}
\begin{proof}
It is easy to see (1) and (2). The fold lemma (Theorem \ref{theorem fold independence}) implies (3) immediately (see Figure \ref{figure B to A}).
\end{proof}

\begin{lemma} \label{lemma small AB}
The following hold:
\[ \indepcpx(A_1) \simeq S^0, \quad \indepcpx(A_2) \simeq S^2\vee S^2, \quad \indepcpx(B_1) \simeq S^1, \quad \indepcpx(B_2) \simeq S^2 \vee S^2. \]
\end{lemma}
\begin{proof}
  The first and third homotopy equivalences are easily deduced from the fold lemma (Theorem \ref{theorem fold independence}).
  To see the second homotopy equivalence, $\indepcpx(A_2 - v_2) = \indepcpx(X_2) \simeq S^2$ and the fold lemma implies $\indepcpx(A_2 - N[v_2]) \simeq S^1$. Thus the inclusion $\indepcpx(A_2 - N[v_2]) \hookrightarrow \indepcpx(A_2 - v_2)$ is null-homotopic and hence Theorem \ref{theorem cofiber} implies that $\indepcpx(A_2) \simeq S^2 \vee S^2$.
To see the last homotopy equivalence, it follows from the fold lemma $\indepcpx(B_2 - v_2) \simeq S^2$ and $\indepcpx(B_2 - N[v_2]) \simeq S^1$. Hence Theorem \ref{theorem cofiber} implies that $\indepcpx(B_2) \simeq S^2 \vee S^2$.
\end{proof}

\begin{proposition} \label{proposition key}
The following hold:
\begin{enumerate}[$(1)$]
\item If $n = 2k + 1$, then $\indepcpx(A_n)$ is homotopy equivalent to a wedge of spheres whose dimension is at most $3k$. If $n = 2k$, then $\indepcpx(A_n)$ is homotopy equivalent to a wedge of spheres whose dimension is at most $3k - 1$.

\item For $n \ge 2$, the inclusion $\indepcpx(A_n - N[v_n]) \hookrightarrow \indepcpx(A_n - v_n)$ is null-homotopic. In particular, we have
\[ \indepcpx(A_n) \simeq \indepcpx(X_n) \vee \Sigma \indepcpx(B_{n-1}).\]

\item If $n = 2k+1$, then $\indepcpx(B_n)$ is homotopy equivalent to a wedge of spheres whose dimension is at most $3k+1$. If $n = 2k$, then $\indepcpx(B_n)$ is homotopy equivalent to a wedge of spheres whose dimension is at most $3k - 1$.

\item For $n \ge 3$, the inclusion $\indepcpx(B_n - N[v_n]) \hookrightarrow \indepcpx(B_n - v_n)$ is null-homotopic. In particular, we have
\[ \indepcpx(B_n) \simeq \indepcpx(Y_n) \vee \Sigma^3 \indepcpx(A_{n-2}).\]
\end{enumerate}
\end{proposition}

\begin{figure}[t]
\begin{picture}(380,140)(0,-20)
\multiput(30,40)(27,18){3}{\circle*{3}}
\multiput(70,0)(27,18){3}{\circle*{3}}
\multiput(70,40)(27,18){3}{\circle*{3}}
\multiput(110,40)(27,18){3}{\circle*{3}}

\multiput(57,18)(27,18){2}{\circle*{3}}
\multiput(137,18)(27,18){2}{\circle*{3}}

\multiput(30,40)(40,0){3}{\line(3,2){65}}
\put(70,0){\line(0,1){40}}
\multiput(57,18)(80,0){2}{\line(3,2){38}}
\put(70,0){\line(3,2){65}}
\multiput(57,18)(27,18){2}{\line(1,0){80}}
\multiput(30, 40)(27,18){3}{\line(1,0){80}}
\multiput(57,18)(40,0){3}{\line(0,1){40}}
\multiput(84,36)(40,0){3}{\line(0,1){40}}

\put(62,45){$v_n$}
\put(80,-23){$A_n$}

\multiput(200,0)(27,18){3}{\circle*{3}}
\multiput(200,40)(27,18){3}{\circle*{3}}
\multiput(240,40)(27,18){3}{\circle*{3}}
\multiput(280,0)(27,18){3}{\circle*{3}}
\multiput(280,40)(27,18){3}{\circle*{3}}
\multiput(267, 18)(27,18){2}{\circle*{3}}

\multiput(200,0)(80,0){2}{\line(3,2){65}}
\multiput(200,40)(40,0){3}{\line(3,2){65}}
\multiput(200,0)(27,18){3}{\line(0,1){40}}
\multiput(280,0)(27,18){3}{\line(0,1){40}}
\multiput(200,40)(27,18){3}{\line(1,0){80}}
\multiput(267,18)(27,18){2}{\line(0,1){40}}
\multiput(227,18)(27,18){2}{\line(1,0){80}}
\multiput(267, 18)(0,40){2}{\line(3,2){38}}

\put(233,45){$v_n$}
\put(250,-23){$B_n$}

\end{picture}
\caption{Graphs $A_n$ and $B_n$} \label{figure AB}
\end{figure}

\begin{figure}[b]
\begin{picture}(480,140)(20,-20)
\multiput(57,58)(27,18){3}{\circle*{3}}
\multiput(30,0)(27,18){4}{\circle*{3}}
\multiput(137,58)(27,18){3}{\circle*{3}}
\multiput(110,0)(27,18){4}{\circle*{3}}
\multiput(124,76)(27,18){2}{\circle*{3}}
\multiput(97,18)(27,18){3}{\circle*{3}}
\multiput(30,40)(40,0){3}{\circle{3}}
\put(70,0){\circle{3}}
\put(97,58){\circle{3}}

\multiput(30,0)(80,0){2}{\line(3,2){95}}
\multiput(84,76)(27,18){2}{\line(1,0){80}}

\multiput(57,58)(80,0){2}{\line(3,2){68}}
\multiput(57,18)(27,18){3}{\line(1,0){80}}
\put(97,18){\line(3,2){68}}
\put(124, 76){\line(3, 2){40}}
\multiput(124, 36)(27,18){2}{\line(0,1){40}}

\multiput(57,18)(27,18){3}{\line(0,1){40}}
\multiput(137,18)(27,18){3}{\line(0,1){40}}

\put(59,32){$v_n$}

\put(80,-23){$B_n - N[v_n]$}


\multiput(250,0)(27,18){2}{\circle*{3}}
\multiput(330,0)(27,18){2}{\circle*{3}}
\multiput(277,58)(80,0){2}{\circle{3}}
\multiput(304,36)(80,0){2}{\circle{3}}
\multiput(304,76)(40,0){3}{\circle*{3}}
\multiput(331,54)(40,0){3}{\circle*{3}}
\multiput(331,94)(40,0){3}{\circle*{3}}
\put(344,36){\circle*{3}}
\put(317,18){\circle{3}}

\multiput(304,76)(27,18){2}{\line(1,0){80}}
\multiput(304,76)(40,0){3}{\line(3,2){38}}
\put(344,36){\line(3,2){38}}
\put(344,36){\line(0,1){40}}
\multiput(331, 54)(40,0){3}{\line(0,1){40}}
\put(331,54){\line(1,0){80}}
\multiput(331,54)(80,0){2}{\line(3,2){11}}

\multiput(250,0)(80,0){2}{\line(3,2){27}}

\put(280,-23){$A_{n-2} \sqcup K_2 \sqcup K_2$}
\end{picture}
\caption{$\indepcpx(Y_n) \simeq \Sigma^3 \indepcpx(Y_{n-2})$} \label{figure B to A}
\end{figure}

\begin{proof}
We prove these four statements simultaneously by induction on $n$. The case $n \le 2$ follows immediately from Lemma \ref{lemma small AB}. We note that in the statement $(2)$ (or (4)), Theorem \ref{theorem cofiber} imply that the former assertion implies the latter.

We now show (1) and (2). Suppose that $n$ is odd, and set $n = 2k + 1$. Then Lemma \ref{lemma X} implies
\[ \indepcpx(A_{2k + 1} - v_{2k + 1}) = \indepcpx(X_{2k + 1}) \simeq \pt.\]
Hence (2) is clear and $\indepcpx(A_{2k + 1}) \simeq \Sigma \indepcpx(A_{2k + 1} - N[v_{2k + 1}]) \cong \Sigma \indepcpx(B_{2k})$. Then, by the inductive hypothesis, $\indepcpx(B_{2k})$ is homotopy equivalent to a wedge of spheres whose dimension is at most $3k - 1$. Thus $\indepcpx(A_{2k + 1})$ is homotopy equivalent to a wedge of spheres whose dimension is at most $3k$.

Next suppose that $n$ is even, and set $n = 2k$. Lemmas \ref{lemma X} and \ref{lemma AB} imply
\[ \indepcpx(A_{2k} - v_{2k}) = \indepcpx(X_{2k}) \simeq S^{3k - 1} \quad \textrm{and} \quad \indepcpx(A_{2k} - N[v_{2k}]) \simeq \indepcpx(B_{2k - 1}).\]
By the inductive hypothesis, $\indepcpx(B_{2k - 1})$ is homotopy equivalent to a wedge of spheres whose dimension is at most $3k - 2$. Hence we have (2) and we have a homotopy equivalence
\[ \indepcpx(A_{2k}) \simeq S^{3k - 1} \vee \Sigma \indepcpx(A_{2k} - N[v_{2k}]).\]
This implies that $\indepcpx(A_{2k})$ is homotopy equivalent to a wedge of spheres whose dimension is at most $3k - 1$. This completes (1).

Next we show (3) and (4). Suppose that $n$ is odd, and set $n = 2k+1$. Then Lemmas \ref{lemma Y} and \ref{lemma AB} imply
\[ \indepcpx(B_{2k+1} - v_{2k+1}) = \indepcpx(Y_{2k+1}) \simeq S^{3k+1} \quad \textrm{and} \quad \indepcpx(B_{2k+1} - N[v_{2k+1}]) \simeq \Sigma^2 \indepcpx(A_{2k-1}).\]
By the inductive hypothesis, $\indepcpx(A_{2k-1}) = \indepcpx(A_{2(k-1) + 1})$ is homotopy equivalent to a wedge of spheres whose dimension is at most $3(k-1) = 3k - 3$. Thus the inclusion $\Sigma^2 \indepcpx(A_{2k-1}) \to \indepcpx(Y_{2k+1})$ is null-homotopic, which implies (4), and
\[ \indepcpx(B_{2k+1}) \simeq \indepcpx(Y_{2k+1}) \vee \Sigma^3 \indepcpx(A_{2k-1}).\]
This implies that $\indepcpx(B_{2k+1})$ is homotopy equivalent to a wedge of spheres whose dimension is at most $3k+1$, which implies (3).

Finally suppose that $n$ is even, and set $n = 2k$. We have already showed the case that $k = 1$ in the previous lemma and its proof. Suppose $k \ge 2$. Then Lemmas \ref{lemma Y} and \ref{lemma AB} imply
\[ \indepcpx(B_{2k} - v_{2k}) = \indepcpx(Y_{2k}) \simeq S^{3k-1} \quad \textrm{and} \quad \indepcpx(B_{2k} - N[v_{2k}]) \simeq \Sigma^2 \indepcpx(A_{2k - 2}).\]
Then $\indepcpx(A_{2k - 2}) = \indepcpx(A_{2(k-1)})$ is homotopy equivalent to a wedge of spheres whose dimension is at most $3(k-1) - 1 = 3k - 4$. This implies that the inclusion $\Sigma^2 \indepcpx(A_{2k-2}) \to \indepcpx(Y_{2k})$ is null-homotopic, which implies (4), and there is a homotopy equivalence
\[ \indepcpx(B_{2k}) \simeq \indepcpx(Y_{2k}) \vee \Sigma^3 \indepcpx(A_{2k - 2}).\]
This implies that $\indepcpx(B_{2k})$ is homotopy equivalent to a wedge of spheres whose dimension is at most $3k-1$. This completes the proof.
\end{proof}

Since the homotopy types of $\indepcpx(X_n)$ and $\indepcpx(Y_n)$ are determined, the homotopy type of $\indepcpx(A_n)$ is recursively determined by $\indepcpx(A_1)$, $\indepcpx(A_2)$, and $\indepcpx(A_3)$.

\begin{corollary} \label{corollary small A}
There are the following homotopy equivalences:
\[ \indepcpx(A_1) \simeq S^0, \quad \indepcpx(A_2) \simeq S^2 \vee S^2, \quad \indepcpx(A_3) \simeq S^3 \vee S^3,\]
\end{corollary}
\begin{proof}
The homotopy equivalences $\indepcpx(A_1) \simeq S^0$ and $\indepcpx(A_2) \simeq S^2 \vee S^2$ were already given in Lemma \ref{lemma small AB}.
Since $\indepcpx(A_3) \simeq \indepcpx(X_3) \vee \Sigma \indepcpx(B_2)$, Lemmas \ref{lemma X} and \ref{lemma small AB} imply $\indepcpx(A_3) \simeq S^3 \vee S^3$.
\end{proof}


The following proposition allows us to determine the homotopy type of $\indepcpx(A_n)$ recursively.

\begin{proposition} \label{proposition inductive}
The following hold:
\begin{enumerate}[$(1)$]
\item For $n \ge 4$, there is the following homotopy equivalence:
\[ \indepcpx(A_n) \simeq \indepcpx(X_n) \vee \Sigma \indepcpx(Y_{n-1}) \vee \Sigma^4 \indepcpx(A_{n-3}).\]

\item If $2k + 1 \ge 7$, there is the following homotopy equivalence:
\[ \indepcpx(A_{2k + 1}) \simeq \Big( \bigvee_3 S^{3k} \Big) \vee \Sigma^{8} \indepcpx(A_{2k+1 - 6}).\]

\item For \(m, k\ge 0\), there is the following homotopy equivalence:
\[ \indepcpx(A_{6m + 2k + 1}) \simeq \bigvee_3 S^{9m + 3k} \vee \cdots \vee \bigvee_3 S^{8m + 3k + 1} \vee \Sigma^{8m} \indepcpx(A_{2k+1})\]
\end{enumerate}
\end{proposition}
\begin{proof}
Suppose $n \ge 4$. Then Proposition \ref{proposition key} implies
\begin{eqnarray*}
\indepcpx(A_n) & \simeq & \indepcpx(X_n) \vee \Sigma \indepcpx(B_{n-1}) \\
& \simeq & \indepcpx(X_n) \vee \Sigma \big( \indepcpx(Y_{n-1}) \vee \Sigma^3 \indepcpx(A_{n-3}) \big) \\
& = & \indepcpx(X_n) \vee \Sigma \indepcpx(Y_{n-1}) \vee \Sigma^4 \indepcpx(A_{n-3}).
\end{eqnarray*}
This completes the proof of (1).

Suppose $2k + 1 \ge 7$. It follows from Lemmas \ref{lemma X} and \ref{lemma Y} and (1) that
\begin{eqnarray*}
\indepcpx(A_{2k+1}) & \simeq & \indepcpx(X_{2k+1}) \vee \Sigma \indepcpx(Y_{2k}) \vee \Sigma^4 \indepcpx(A_{2k - 2}) \\
& \simeq & S^{3k} \vee \Sigma^4 (\indepcpx(X_{2k - 2}) \vee \Sigma \indepcpx(Y_{2k - 3}) \vee \Sigma^4 \indepcpx(A_{2k - 5})) \\
& \simeq & S^{3k} \vee S^{3k} \vee S^{3k} \vee \Sigma^8 \indepcpx(A_{2k + 1 - 6}).
\end{eqnarray*}
if $2k+1 \ge 7$. We deduce (3) by iterating (2).
\end{proof}

In case $n$ is small the homotopy type of $\indepcpx(A_n)$ is determined as follows.

\begin{lemma} \label{lemma small A}
There are the following homotopy equivalences:
\[ \indepcpx(A_1) \simeq S^0, \quad \indepcpx(A_2) \simeq S^2 \vee S^2, \quad \indepcpx(A_3) \simeq S^3 \vee S^3,\]
\[ \indepcpx(A_4) \simeq S^5 \vee S^5 \vee S^4, \quad \indepcpx(A_5) \simeq S^6 \vee S^6 \vee S^6, \quad \indepcpx(A_6) \simeq S^8 \vee S^8 \vee S^7 \vee S^7.\]
\end{lemma}
\begin{proof}
The homotopy equivalence $\indepcpx(A_1) \simeq S^0$ was already proved (Lemma \ref{lemma small AB}). By Proposition \ref{proposition key} and Lemma \ref{lemma X}, we have
\[ \indepcpx(A_2) \simeq \indepcpx(X_2) \vee \Sigma \indepcpx(B_1) \simeq S^2 \vee S^2.\]
Proposition \ref{proposition key} and Lemmas \ref{lemma X} and \ref{lemma small AB}, we have
\[ \indepcpx(A_3) \simeq \indepcpx(X_3) \vee \Sigma \indepcpx(B_2) \simeq \Sigma \indepcpx(B_2) \simeq S^3 \vee S^3.\]
The other homotopy equivalences are deduced by Lemmas \ref{lemma X}, \ref{lemma Y} and (1) of Proposition \ref{proposition inductive}.
\end{proof}

In the following two propositions we determine the homotopy type of $I(A_n)$.

\begin{proposition} \label{proposition odd A}
For $m \ge 0$, there are the following homotopy equivalences:
\[ \indepcpx(A_{6m + 1}) \simeq \Big( \bigvee_3 S^{9m} \vee \cdots \vee \bigvee_3 S^{8m + 1} \Big) \vee S^{8m}, \]
\[ \indepcpx(A_{6m + 3}) \simeq \Big( \bigvee_3 S^{9m + 3} \vee \cdots \vee \bigvee_3 S^{8m + 4} \Big) \vee \bigvee_2 S^{8m + 3}, \]
\[ \indepcpx(A_{6m + 5}) \simeq \bigvee_3 S^{9m + 6} \vee \cdots \vee \bigvee_3 S^{8m + 6}.\]
\end{proposition}
\begin{proof}
This is deduced from Lemma \ref{lemma small A} and (3) of Proposition \ref{proposition inductive}.
\end{proof}

\begin{proposition} \label{proposition even A}
For $m \ge 0$, there are the following homotopy equivalences:
\[ \indepcpx(A_{6m}) \simeq \bigvee_2 S^{9m - 1} \vee \Big( \bigvee_3 S^{9m -2} \vee \cdots \vee \bigvee_3 S^{8m} \Big) \vee \bigvee_2 S^{8m - 1} \quad (m \ge 1),\]
\[ \indepcpx(A_{6m + 2}) \simeq \bigvee_2 S^{9m + 2} \vee \Big( \bigvee_3 S^{9m + 1} \vee \cdots \vee \bigvee_3 S^{8m + 2} \Big)\]
\[ \indepcpx(A_{6m + 4}) \simeq \bigvee_2 S^{9m + 5} \vee \Big( \bigvee_3 S^{9m + 4} \vee \cdots \vee \bigvee_3 S^{8m + 5} \Big) \vee S^{8m + 4}.\]
\end{proposition}
\begin{proof}
This is deduced from Lemmas \ref{lemma X}, \ref{lemma Y}, \ref{lemma small A}, Proposition \ref{proposition odd A} and (1) of Proposition \ref{proposition inductive}.
\end{proof}

\subsection{Homotopy type of \(\indepcpx(\Gamma_n)\) and homology of \(\domcpx(P_n\times P_3)\)}

In this subsection we determine the homotopy type of $\indepcpx(\Gamma_n)$ and the homology groups of $\domcpx(P_n \times P_3)$. The following proposition shows that the homotopy type of $\indepcpx(\Gamma_n)$ is determined by those of $\indepcpx(A_{n-2})$ and $\indepcpx(Y_n)$.

\begin{proposition} \label{proposition key 2}
For $n \ge 3$, we have the following homotopy equivalence:
\[ \indepcpx(\Gamma_n) \simeq \indepcpx(Y_n) \vee \bigvee_2 \Sigma^3 \indepcpx(A_{n-2}).\]
\end{proposition}
\begin{proof}
Set $w_n = (n,2,1)$. We first show that the inclusion $\indepcpx(\Gamma_n - N[w_n]) \hookrightarrow \indepcpx(\Gamma_n - w_n)$ is null-homotopic. Note that the inclusion $\indepcpx(\Gamma_n - N[w_n]) \hookrightarrow \indepcpx(\Gamma_n - w_n)$ has the following factorization:
\[ \indepcpx(\Gamma_n - N[w_n]) \to \indepcpx(Y_n) \to \indepcpx(\Gamma_n - w_n).\]
Hence we show that the inclusion $\indepcpx(\Gamma_n - N[w_n]) \hookrightarrow \indepcpx(Y_n)$ is null-homotopic.

Let $\alpha$ be the automorphism of $\Gamma_n$ sending a vertex $(i,j,k)$ to the vertex $(i, j, 3 - k)$. Then $\alpha$ restricts to isomorphisms $\Gamma_n - N[w_n]\xrightarrow{\cong}B_n - N[v_n]$ and $Y_n \xrightarrow{\cong}Y_n$. Hence it suffices to see that the inclusion $\indepcpx(\Gamma_n - N[v_n])\hookrightarrow \indepcpx(Y_n)$ is null-homotopic. Since
\[ \indepcpx(B_n - N[v_n]) = \indepcpx(\Gamma_n - N[v_n]) \hookrightarrow \indepcpx(Y_n) = \indepcpx(B_n - v_n),\]
Proposition \ref{proposition key} implies that the inclusion $\indepcpx(\Gamma_n - N[v_n]) \hookrightarrow \indepcpx(Y_n)$ is null-homotopic. This concludes that the inclusion $\indepcpx(\Gamma_n - N[w_n]) \hookrightarrow \indepcpx(Y_n)$ is null-homotopic, as desired.

Theorem \ref{theorem cofiber} implies that $\indepcpx(\Gamma_n) \simeq \indepcpx(\Gamma_n - w_n) \vee \Sigma \indepcpx(\Gamma_n -N[w_n])$. Lemma \ref{lemma AB} implies that $\indepcpx(\Gamma_n - N[w_n]) \cong \indepcpx(B_n - N[v_n]) \simeq \Sigma^2 \indepcpx(A_{n-2})$. Since $\Gamma_n - w_n = B_n$, (4) of Proposition \ref{proposition key} implies
\[ \indepcpx(\Gamma_n) \simeq \indepcpx(\Gamma_n - w_n) \vee \Sigma \indepcpx(\Gamma_n - N[w_n]) \simeq \indepcpx(Y_n) \vee \bigvee_2 \Sigma^3 \indepcpx(A_{n-2}).\]
This completes the proof.
\end{proof}

We are now ready to determine the homotopy type of $\indepcpx(\Gamma_n)$ completely.

\begin{theorem}\label{theorem Gamma}
There are the following homotopy equivalences:
\[ \indepcpx(\Gamma_{6m}) \simeq \bigvee_5 S^{9m-1} \vee \Big( \bigvee_6 S^{9m - 2} \vee \cdots \vee \bigvee_6 S^{8m} \Big) \vee \bigvee_2 S^{8m-1} \quad (m \ge 1),\]
\[ \indepcpx(\Gamma_{6m + 1}) \simeq S^{9m+1} \vee \Big( \bigvee_6 S^{9m} \vee \cdots \vee \bigvee_6 S^{8m + 1} \Big),\]
\[ \indepcpx(\Gamma_{6m + 2}) \simeq \begin{cases}
S^2 \vee S^2 \vee S^2 & (m = 0) \\
\bigvee_5 S^{9m + 2} \vee \Big( \bigvee_6 S^{9m + 1} \vee \cdots \vee \bigvee_6 S^{8m + 3} \Big) \vee \bigvee_4 S^{8m + 2} & (m \ge 1),
\end{cases}\]
\[ \indepcpx(\Gamma_{6m + 3}) \simeq S^{9m + 4} \vee \Big( \bigvee_6 S^{9m + 3} \vee \cdots \vee \bigvee_6 S^{8m + 4} \Big) \vee \bigvee_2 S^{8m+3},\]
\[ \indepcpx(\Gamma_{6m + 4}) \simeq \bigvee_5 S^{9m + 5} \vee \Big( \bigvee_6 S^{9m + 4} \vee \cdots \bigvee_6 S^{8m+5} \Big),\]
\[ \indepcpx(\Gamma_{6m + 5}) \simeq S^{9m + 7} \vee \Big( \bigvee_6 S^{9m + 6} \vee \cdots \vee \bigvee_6 S^{8m + 7} \Big) \vee \bigvee_4 S^{8m + 6}.\]
\end{theorem}
\begin{proof}
The fold lemma (Theorem \ref{theorem fold independence}) implies $\indepcpx(\Gamma_1) \simeq S^1$. Now we show $\indepcpx(\Gamma_2) \simeq S^2 \vee S^2 \vee S^2$. The fold lemma implies $\indepcpx(\Gamma_2 - N[v_2]) \simeq S^1$. Since $\indepcpx(\Gamma_2 - v_2) \cong \indepcpx(B_2) \simeq S^2 \vee S^2$, the inclusion $\indepcpx(\Gamma_2 - N[v_2]) \hookrightarrow \indepcpx(\Gamma_ 2 - v_2)$ is null-homotopic. Hence Theorem \ref{theorem cofiber} implies
\[ \indepcpx(\Gamma_2) \simeq \indepcpx(\Gamma_2 - v_2) \vee \Sigma \indepcpx(\Gamma_2 - N[v_2]) \simeq S^2 \vee S^2 \vee S^2.\]
The other homotopy equivalences are deduced from Propositions \ref{proposition key 2}, \ref{proposition odd A} and \ref{proposition even A}.
\end{proof}

We conclude this paper by determining the homology groups of $\domcpx(P_n \times P_3)$ as follows.

\begin{theorem}
The homology group of $\domcpx(P_n \times P_3)$ is described as follows:
\begin{enumerate}[$(1)$]
\item If $n = 6m$, then
\[ \tilde{H}_i (\domcpx(P_{6m} \times P_3)) \cong \begin{cases}
\ZZ^2 & (i = 10m - 1) \\
\ZZ^6 & (9m - 1 < i < 10 m - 1) \\
\ZZ^5 & (i = 9m-1) \\
0 & (\textrm{otherwise}).
\end{cases}\]

\item If $n = 6m + 1$, then
\[ \tilde{H}_i(\domcpx(P_{6m + 1} \times P_3)) \cong \begin{cases}
\ZZ^6 & (9m < i \le 10m) \\
\ZZ & (i = 9m) \\
0 & (\textrm{otherwise}).
\end{cases}\]

\item If $n = 6m + 2$ and $m \ge 1$, then
\[ \tilde{H}_i (\domcpx(P_{6m + 2} \times P_3)) \cong \begin{cases}
\ZZ^4 & (i = 10m + 2)\\
\ZZ^6 & (9m + 2 < i < 10m + 2) \\
\ZZ^5 & (i = 9m + 2) \\
0 & (\textrm{otherwise}).
\end{cases}\]
If $n = 2$, then
\[ \tilde{H}_i (\domcpx(P_2 \times P_ 3)) \cong \begin{cases}
\ZZ^3 & (i = 2)\\
0 & (\textrm{otherwise}).
\end{cases}\]

\item If $n = 6m + 3$, then
\[ \tilde{H}_i (\domcpx(P_{6m + 3} \times P_3)) \cong \begin{cases}
\ZZ^2 & (i = 10 m + 4) \\
\ZZ^6 & (9m + 3 < i < 10m + 4) \\
\ZZ & (i = 9m + 3) \\
0 & (\textrm{otherwise}).
\end{cases}\]

\item If $n = 6m + 4$, then
\[ \tilde{H}_i (\domcpx(P_{6m + 4} \times P_3)) \cong \begin{cases}
\ZZ^6 & (9m + 5 < i \le 10m + 5) \\
\ZZ^5 & (i = 9m + 5) \\
0 & (\textrm{otherwise}).
\end{cases}\]

\item If $n = 6m + 5$, then
\[ \tilde{H}_i (\domcpx(P_{6m + 5} \times P_3)) \cong \begin{cases}
\ZZ^4 & (i = 10m + 7) \\
\ZZ^6 & (9m + 6 < i < 10m + 7) \\
\ZZ & (i = 9m + 6) \\
0 & (\textrm{otherwise}).
\end{cases}\]
\end{enumerate}
\end{theorem}

\begin{proof}
  By Proposition \ref{proposition suspension}, Theorem \ref{theorem alexander} and Theorem \ref{theorem bowtie},
  we have an isomorphism
  \begin{equation}
    \tilde{H}_i(\domcpx(P_n\times P_3)) \cong
    \tilde{H}^{3n-i-2}(\indepcpx(\Gamma_n)).
  \end{equation}
  Hence Proposition \ref{proposition homology of wedges of spheres} and Theorem \ref{theorem Gamma} complete the proof.
\end{proof}

\section*{Acknowledgment}
The first author was supported in part by JSPS KAKENHI Grant Numbers JP19K14536 and JP23K12975. The second author was supported in part by JSPS KAKENHI Grant Number JP23K19006. The authors would like to express their sincere gratitude to the anonymous referees for their constructive suggestions and comments, which have improved the description of the paper.


\bibliographystyle{abbrvurl}
\bibliography{references}

\end{document}